\documentclass[11pt,a4paper,fleqn]{article}

\oddsidemargin=.\paperwidth
\topmargin=.\paperheight

\makeatletter
\newcommand\footnotetext@\relax
\let\footnotetext@\@footnotetext
\newcommand{\MSC}[2][2020]{%
 \unskip\protected@xdef\@thefnmark{}%
 \protect\footnotetext@{\kern-1.8em{\itshape MSC#1\spacefactor3000:}\/ #2}}
\newcommand{\keywords}[1]{%
 \unskip\protected@xdef\@thefnmark{}%
 \protect\footnotetext@{\kern-1.8em{\itshape Keywords\spacefactor3000:}\/ #1}}
\newcommand{\address}[1]{%
 \unskip\footnotemark
 \protected@xdef\@thanks{\@thanks\protect\footnotetext[\the\c@footnote]{{\itshape Address\spacefactor3000:}\/ #1}}}
\newcommand{\email}[1]{%
 \unskip\protected@xdef\@thanks{\@thanks\protect\footnotetext[0]{{\itshape Email\spacefactor3000:}\/ \texttt{#1}}}}
\renewcommand{\thanks}[1]{%
 \unskip\protected@xdef\@thanks{\@thanks\protect\footnotetext[0]{#1}}}
\gdef\@date{}
\makeatother

\usepackage[english]{babel}
\usepackage[utf8]{inputenc}
\usepackage[T1]{fontenc}
\usepackage{amsmath}%
 \allowdisplaybreaks[4]
\usepackage{amssymb,amsthm}%
 \swapnumbers
 \theoremstyle{plain}
  \newtheorem{thm}{Theorem}
  \newtheorem{cor}[thm]{Corollary}
  \newtheorem{lem}[thm]{Lemma}
  \newtheorem{prop}[thm]{Proposition}
 \theoremstyle{definition}
  \newtheorem{defn}[thm]{Definition}
  \newtheorem{exm}[thm]{Example}
  \newtheorem{exms}[thm]{Examples}
  \newtheorem*{exm*}{Example}
  \newtheorem*{exms*}{Examples}
  \newtheorem*{LEP}{Local Extension Problem}
  \newtheorem*{LGP}{Local-to-Global Problem}
 \theoremstyle{remark}
  
  \newtheorem{rmks}[thm]{Remarks}
  \newtheorem*{rmk*}{Remark}
  \newtheorem*{rmks*}{Remarks}
  \newtheorem*{cmt*}{Comment}
  \newtheorem*{cmts*}{Comments}
\usepackage[all]{xy}%
 \entrymodifiers={!!<.em,.6ex>+}
 \SelectTips{cm}{11}
\usepackage{hyperref}

\newcommand{\append}[3][]{(#3)#2{\,\vphantom{#1}}\mkern+2mu}
\DeclareMathOperator{\Aut}{Aut}
\newcommand{\avg}{\hat}
\newcommand{\Ban}{\boldsymbol}
\newcommand{\BB}{\Cat{BB}}
\DeclareMathOperator{\Bid}{Bid}
\newcommand{\blank}{\hphantom{b}}
\newcommand{\C}{\mathbb{C}}
\newcommand{\Cat}{\mathsf}
\newcommand{\cat}{\mathcal}
\newcommand{\clos}{\overline}
\newcommand{\cnj}{\overline}
\newcommand{\der}{\mathit{d}}
\DeclareMathOperator{\End}{End}

\newcommand{\ftimes}[2]{\mathbin{_{#1}\times_{#2}}}
\DeclareMathOperator{\GI}{I}
\DeclareMathOperator{\GL}{GL}
\DeclareMathOperator{\GU}{U}
\newcommand{\HB}{\Cat{HB}}
\newcommand{\Hil}{\Cat{Hil}}
\DeclareMathOperator{\Hom}{Hom}
\newcommand{\hto}{\Rightarrow}
\newcommand{\id}{\mathinner{\mathrm{id}}}
\def\iintegral#1\der#2\der{\iint#1\mathclose{\thinspace}\der#2\thinspace\der}
\def\integral#1\der{\int#1\mathclose{\thinspace}\der}
\DeclareMathOperator{\Iso}{Iso}
\newcommand{\justify}[1]{\mathrel{\phantom=}#1\mathopen{\mkern\medmuskip}}
\renewcommand{\L}{\operatorname{L}}
\newcommand{\longfrom}{\longleftarrow}
\newcommand{\longmapsfrom}{\longleftarrow\mapsfromchar}
\newcommand{\longto}{\longrightarrow}
\DeclareSymbolFont{stmry}{U}{stmry}{m}{n}
\DeclareMathSymbol{\mapsfromchar}{\mathrel}{stmry}{"5B}

\newcommand{\pr}{\mathinner{\mathrm{pr}}}
\newcommand{\R}{\mathbb{R}}
\newcommand{\Rep}{\Cat{Rep}}
\newcommand{\simto}{\overset\sim\to}
\DeclareMathOperator{\supp}{supp}
\DeclareMathOperator{\stab}{stab}
\newcommand{\tto}{\rightrightarrows}
\newcommand{\VSp}{\Cat{Vec}}
\newcommand{\xfrom}{\xleftarrow}
\newcommand{\xmapsfrom}[1]{\xfrom{#1}\mapsfromchar}
\newcommand{\xmapsto}[1]{\mapstochar\xto{#1}}
\newcommand{\xto}{\xrightarrow}

\begin{document}

\title{Almost Representations of Groupoids on Banach Bundles%
  \MSC{Primary 22A22, Secondary 22A30, 28C10}
  \keywords{Locally compact groupoid, Banach bundle, continuous representation, %
      Peter–Weyl theorem, Tannaka duality}
}%
\author{Giorgio Trentinaglia%
  \address{Centro de Análise Matemática, Geometria e Sistemas Dinâmicos, %
      Ins\-ti\-tu\-to Su\-pe\-ri\-or Téc\-ni\-co, University of Lisbon, %
      Av.~Ro\-vis\-co Pais, 1049-001 Lisbon, Portugal}
  \email{gtrentin@math.tecnico.ulisboa.pt}
  \thanks{The author acknowledges the support %
      of the Portuguese Foundation for Science and Technology %
      through grants SFRH/BPD/81810/2011 and UID/MAT/04459/2020.}
}%
\maketitle

\begin{abstract}
We extend an old result of de~la~Harpe and Karoubi, concerning almost representations of compact groups, to proper groupoids admitting continuous Haar measure systems. As an application, we establish the existence of sufficiently many continuous representations of such groupoids on finite-rank Hilbert bundles locally, and use this fact to prove a new generalization of the classical Tannaka duality theorem to groupoids in a purely topological setting.
\end{abstract}

\section{Introduction}\label{sec:intro}

The Peter–Weyl and the Tannaka duality theorems are cornerstones in the representation theory of compact groups \cite{BtD85,JS91}. When one attempts to generalize these classical results to groupoids, one runs into notorious difficulties, among which the problem of determining whether the \emph{representative functions}---the “matrix coefficients” of finite-dimensional continuous representations---are enough to separate points has remained unsettled for nearly two decades \cite{Ami07,Bos11} or arguably even longer \cite[pp.~152–170]{Sed74}. Allowing ourselves a bit of heuristic license, we may break the latter problem into the following two questions:

\begin{LEP} Let $x$ be a base point of a topological groupoid $G$: under what general circumstances does a continuous representation of the stabilizer group of $x$ extend continuously to a local representation of $G$ defined in an open neighborhood of $x$? \end{LEP}

\begin{LGP} Given any local representation of $G$ operating continuously in an open neighborhood of $x$, when does it agree (in a possibly smaller neighborhood) with the restriction of some continuous representation of the whole $G$? \end{LGP}

Let us briefly review what is known about each one of these two problems, beginning with the second.

The answer to the local-to-global problem is rather sensitive to what precise meaning one attaches to the term “continuous representation”. As a matter of principle, just as groups act linearly on vector spaces, so should groupoids operate by fiberwise linear transformations on vector bundles: but what kind of vector bundles, exactly? Now, examples indicate that representative functions arising from continuous representations of compact groupoids on locally trivial vector bundles of locally finite rank are generally too few to separate points \cite{JM15,LO01,Mrc18,2008b}. The notion of election is thus ruled out and we are faced with the dilemma of deciding which of the defining conditions should be relaxed. With an eye to keeping as close as possible to the traditional formulation of the classical theorems mentioned at the beginning, giving up the requirement of local triviality seems to be the most sensible option. (Surprising as it may sound, there are reasons of not only \emph{global} but also \emph{local} nature for doing away with that requirement, some of which are discussed in the appendix to this paper; a lesson to be learned is that it is wrong to impose “transversal” rigidity conditions on the objects groupoids are supposed to act on.) Fell's concept of \emph{Banach bundle} \cite{Fel77} proves to be an excellent candidate for the notion of a vector bundle that is not locally trivial but is still, in a suitable sense, “continuous”; as we shall see, working with Banach bundles removes any obstruction to solving the local-to-global problem, yet retains enough topological information for it to be still meaningful to think about generalizing the aforesaid theorems.

Even after deciding to work with continuous representations of $G$ on Banach bundles (or with an appropriate selection thereof), we are still faced with the local extension problem. In the literature the latter problem is commonly sidestepped either by postulating without further analysis as a working hypothesis on $G$ that the desired extensions always exist \cite{Ami07,Bos11} or by assuming even more restrictively that $G$ is an “orbispace”, i.e., as a topological stack, it is locally equivalent in the neighborhood of each point to the quotient of some locally compact Hausdorff space by a continuous action of some compact group. This state of affairs is quite unsatisfactory. When $G$ is proper and has a \emph{differentiable} rather than just topological structure and so is a \emph{Lie groupoid}, it must be an orbispace in this sense \cite{Wei00,Zun06}; however, in the purely topological setting of locally compact groupoids, which are conceivably as far removed from being differentiable as to be totally disconnected, this may no longer be the case even when $G$ arises from a proper continuous action of a noncompact locally compact group. To our knowledge, no general solution to the local extension problem has been provided to date for such $G$, or for those “continuous bundles” of compact groups which are not locally trivial, just to mention a couple of basic examples.

The principal contribution of the present paper is a complete (and positive) solution to the local extension problem under no other hypothesis on $G$ besides that it be a \emph{proper} (\emph{locally compact}, \emph{topological}) \emph{groupoid}; here, the concept of properness includes the existence of continuous normalized Haar measure systems as a built-in feature. Our solution takes the form of a \emph{local extension theorem} for isotropy representations of proper groupoids, stated as Theorem \ref{thm:local} on page \pageref{thm:local}. This result, whose proof takes up much of section \ref{sec:local}, is a corollary of what we consider to be the second most significant contribution of our paper, an \emph{almost representation theorem} for pseudorepresentations of proper groupoids on Banach bundles, which is enunciated at the beginning of section \ref{sec:almost} on page \pageref{thm:almost} as Theorem \ref{thm:almost} after a self-contained review of the basic notions involved in the formulation of both theorems which is the subject of section \ref{sec:representations}. Loosely speaking, our almost representation theorem says that, for a proper groupoid $G$ (admitting continuous Haar measure systems), any “almost representation” of $G$ on a Banach bundle lies “close” to an actual representation; while it can be seen as a far-reaching generalization of an old result of de~la~Harpe and Karoubi about compact groups \cite{dlHK77}, our theorem is given a completely different proof which is not only more general but also much shorter and makes our contribution of independent interest even in the case of compact groups. In section \ref{sec:tannakian}, as an application of our local extension theorem, we formulate and prove a \emph{tannakian reconstruction theorem} for general (not necessarily \emph{Lie}) proper groupoids; this is Theorem \ref{thm:tannakian} on page \pageref{thm:tannakian}, our third main contribution. In order to leave the door open for other applications, in both Theorem \ref{thm:almost} and Theorem \ref{thm:local} we deal at practically no extra cost with \emph{projective}, as well as \emph{infinite-dimensional}, representations. Finally, our discussion in section \ref{sec:images} of direct images of projective representations, which is needed in the proof of our local extension theorem, may also be of independent interest, and perhaps, even contain original aspects.

We want to conclude the present introduction with a comparative-historical perspective of the ideas involved in the demonstration of our almost representation theorem. Our use of the analytic technique of “recursive averaging” sets our proof apart from that of de~la~Harpe and Karoubi \cite{dlHK77}, who employ different methods including “holomorphic functional calculus” and the concept of regular representation with coefficients in a Banach module. Although the idea of recursive averaging is not new, its practical implementation in the present work features original aspects. A good way of illustrating their extent is to contrast our argument with the one put forward by Kazhdan in \cite{Kaz82}. Kazhdan's reasoning may be summarized (somewhat informally) as follows. If $T: G \to \GU(\Ban{E})$ is a uniformly continuous map of a compact group into the group of unitary operators in a Hilbert space, then setting \[%
	W(g,h) = T(g)T(h)T(gh)^{-1}
\] for all $g$,~$h \in G$ defines a “nonabelian multiplier” for $G$ with values in $\GU(\Ban{E})$, where $g$ “acts” on $\GU(\Ban{E})$ via conjugation by $T(g)$. Suppose $W(g,h)$ lies “close enough” to $I$, the identity operator, for all $g$ and $h$. Were $\exp$, the exponential map, a homomorphism from the additive group of all bounded skew-adjoint operators in $\Ban{E}$ to $\GU(\Ban{E})$, we could use its local inverse, $\log$, defined around $I$, to turn $W$ into a continuous \emph{additive} $2$-cocycle for $G$. Were furthermore the adjoint operation of $G$ on $\End(\Ban{E})$ via $T$ an action, rather than merely a “pseudo-action”, we could take advantage of the compactness of $G$ in the form of a standard cohomology vanishing argument in order to find a $1$-cochain on $G$ whose image under $\exp$ could then be used to “correct” $T$ so as to turn it into an actual representation. Although this heuristic reasoning has hardly any practical applicability, it is nevertheless “approximately” valid, in a sense that can be made precise, and $T$ can be “corrected” \emph{recursively} so as to produce a true representation \emph{in the limit}. Essentially the same idea is implicit in \cite{GKR74,Zun06}, with the difference that these works consider “almost homomorphisms” $T: G \to H$ with values in a compact Lie group $H$ rather than in $\GU(\Ban{E})$: this introduces an extra layer of technical complications coming from the need to get good estimates for $\exp: \mathfrak{h} \to H$, where $\mathfrak{h}$ is the Lie algebra of $H$; like us, Zung \cite{Zun06} deals more generally with the case where $G$ is a groupoid, rather than a group. Now, while it is conceivable that the theory of \cite{Kaz82} might be extended without too much difficulty to (nonunitary, nonuniformly continuous, projective) representations of groupoids on Banach bundles, we need not take the pains to do that, for a simpler alternative is available. As a matter of fact, the reader will not find any use of “$\varepsilon$-approximate” groupoid cohomology or any vestiges of $\exp$ or $\log$ in the proof of our almost representation theorem in section \ref{sec:almost}. Instead, we shall exploit the simple observation that $\GL(\Ban{E})$ already sits in a vector space, $\End(\Ban{E})$, within which it makes sense to take integral averages. In this way, we most likely attain the greatest possible concision: apart from a technical lemma that is needed anyway to deal with the nonunitary case, and apart from the unavoidable preliminaries about normalized Haar integrals, our proof at its heart boils down to a single one-page lemma. We finish with one brief comment. The use of recursive averaging techniques in the theory of groupoids was advocated by Weinstein \cite{Wei00} as a possible means of proving the linearization theorem for proper Lie groupoids. Subsequent to Zung's implementation of Weinstein's program \cite{Zun06}, the idea of recursive averaging has gradually fallen into disfavor among differential geometers, mostly because of the discovery of new proofs of the linearization theorem based on “geometric” rather than “analytic” methods. Still, as the present work indicates, the same idea remains very useful in certain contexts, such as the theory of topological groupoids, where geometric concepts do not apply.

\section{Representations of locally compact groupoids}\label{sec:representations}

The present section is aimed at nonspecialists: its purpose is to provide just enough background on locally compact groupoids and their Banach bundle representations to get started.

\paragraph*{Locally compact groupoids.} Let $G \tto X$ be a topological groupoid, with arrows $g$,~$h$,~$k \in G$, base points $x$,~$y$,~$z \in X$, continuous source and target map $s$,~$t: G \to X$, space of composable pairs $G \ftimes{s}{t} G = \{(g,h): sg = th\} \subset G \times G$, continuous composition $G \ftimes{s}{t} G \to G$, $(g,h) \mapsto gh$, and continuous inversion $G \to G$, $g \mapsto g^{-1}$. The topology of $X$ is determined uniquely by that of $G$: it is the smallest making the unit map $1: X \to G$, $x \mapsto 1x$ continuous, equivalently, the largest making the source map $s$ (or target map $t$) continuous. We write $G_x = s^{-1}(x)$ (source fiber), $G^x = t^{-1}(x)$ (target fiber), $G(x,y) = G_x \cap G^y$ (hom-space), $G(x) = G(x,x)$ (isotropy group), and $Gx = t(G_x)$ (orbit).

We assume $G \tto X$ is \emph{locally compact} in the sense that $G$, and hence $X$, is a locally compact \emph{Hausdorff} space; our terminology is consistent with \cite{Ren80}, but diverges slightly from other sources \cite{Pat99}. As in [loc.~cit.], we agree that $G \tto X$ possesses a (continuous, left-invariant) \emph{Haar measure system} $\{\mu^x$,~$x \in X\}$. This means that every $\mu^x$ is a $\sigma$-regular Borel measure (Radon measure) on $G$ and that the following three conditions are satisfied, where we write $\integral \varphi(h) \der\mu^x(h)$ or $\integral_{th=x} \varphi(h) \der h$ for the integral of a compactly supported continuous function $\varphi \in C_c(G)$: (i)~the support of $\mu^x$ is the target fiber $G^x$; (ii)~(left invariance) $\integral \varphi(gh) \der\mu^{sg}(h) = \integral \varphi(h) \der\mu^{tg}(h)$ for all $g \in G$, $\varphi \in C_c(G)$; (iii)~(continuity) the function $x \mapsto \integral \varphi(h) \der\mu^x(h)$ is continuous on $X$ for each $\varphi \in C_c(G)$. The existence of Haar measure systems, which entails the openness of the two maps $s$ and $t$, is commonly postulated to be part of the concept of “locally compact groupoid”; see \cite[§2.2]{Pat99}.

\begin{exms*} {\itshape (a)}\spacefactor3000\/ Any locally compact group $G$ with left Haar measure $\mu$ may be seen as a locally compact groupoid over the one-point space $X = \{\ast\}$.

{\itshape (b)}\spacefactor3000\/ Any locally compact group $G$ acting continuously from the right on a locally compact Hausdorff space $X$ gives rise to a corresponding \emph{transformation groupoid} $X \times G \tto X$: the action $(x,g) \mapsto xg$ functions as the source, the projection $(x,g) \mapsto x$ as the target, and composition is given by $(x,g) \cdot (xg,h) = (x,gh)$. This is a locally compact groupoid, a Haar measure system being furnished by $\mu^x = \delta_x \times \mu$, where $\delta_x$ is Dirac measure at $x$ (unit point-mass) and $\mu$ is as in Example (a). We have a similar transformation groupoid $G \times X \tto X$ associated with any left action of $G$ on $X$. It is a simple matter to check this too is a locally compact groupoid.

{\itshape (c)}\spacefactor3000\/ Let $G \tto X$ be a topological groupoid which is \emph{isotropic} in the sense that its source map and target map coincide. Suppose that $G$ is a locally compact Hausdorff space. Suppose in addition that $s = t$ is an open map and that every isotropy group $G_x = G^x = G(x)$ is compact. Let $\mu^x$ be left Haar measure on $G_x$ normalized so that $\mu^x(G_x) = 1$. Then $\{\mu^x\}$ is a (continuous) Haar measure system on $G \tto X$; a more general statement, which the reader can find either in the appendix [section \ref{sec:appendix}] or in \cite[Lem.~1.3]{Ren91}, is valid for possibly noncompact $G_x$, in which generality we refer to $G \tto X$ as a \emph{locally compact group bundle}.

{\itshape (d)}\spacefactor3000\/ A \emph{Lie groupoid} $G \tto X$ is the result of endowing each one of $G$,\ $X$ with a differentiable structure (of class $C^\infty$, finite-dimensional, Hausdorff, and paracompact) in such a way that $s$,\ $t$ become submersive, as well as differentiable, (whence $G \ftimes{s}{t} G$ becomes a submanifold of $G \times G$) and likewise do the composition and the inversion. A Lie $G \tto X$ is always locally compact in the sense we have stipulated, a Haar measure system being provided by the Riemannian volumes $\mu^x$ of target fibers under any left-invariant metric on the subbundle $\ker Tt$ of $TG$~(= the tangent bundle of $G$), such metrics being always available thanks to the existence of partitions of unity on $X$.

{\itshape (e)}\spacefactor3000\/ An \emph{étale groupoid} is a topological groupoid whose target map $t$ is a local homeomorphism. Any étale $G \tto X$ with locally compact Hausdorff $G$ is a locally compact groupoid, as can be seen by taking each $\mu^x$ to be counting measure on the discrete space $G^x$. Among all groupoids of the latter kind, many important ones, such as the Cuntz or the Wiener–Hopf groupoids \cite[§4.2]{Pat99}, happen to be both totally disconnected and nondiscrete. \end{exms*}

For a prosecution of the present discussion the reader is referred to the appendix [section \ref{sec:appendix}], where additional motivation and examples are made available.

\paragraph*{Banach bundles.} By a \emph{topological group bundle} we mean a topological groupoid whose base space is Hausdorff and whose source map is open and equal to its target map; in this context, we refer to $s = t$ as the \emph{bundle projection}, and denote it by $p$. A \emph{topological vector bundle} is an abelian topological group bundle $E \to X$ (noted additively) equipped with a continuous multiplication by scalars $\R \times E \to E$. (We work with real coefficients.) Following \cite[§10, p.~99]{Fel77}, we define a \emph{Banach bundle} over $X$ to be a Hausdorff topological vector bundle $E \to X$ equipped with a continuous norm $\lvert\blank\rvert: E \to \R_{\geq 0}$ making each fiber $E_x = p^{-1}(x)$ into a Banach space in such a way that if $\{e_i\}$ is any net of elements of $E$ for which $pe_i \to x$ in $X$ and $\lvert e_i\rvert \to 0$ then $e_i \to 0x$~(= the zero vector in $E_x$). The subspace topology on $E_x$ turns out to agree with the norm topology \cite[p.~102, Prop.~10.2]{Fel77}.

Let $\Gamma(X;E)$ be the vector space of all continuous cross-sections of the topological vector bundle $E \to X$. We say that $E \to X$ has \emph{enough cross-sections} if for each $e$ in $E$ there exists at least one element of $\Gamma(X;E)$ passing through $e$. An important result of Douady and dal~Soglio-Hérault says that if $X$ is either paracompact or locally compact then every Banach bundle over $X$ has enough cross-sections \cite[p.~103, Thm.~10.5]{Fel77}.

A \emph{morphism} of Banach (or, more generally, topological vector) bundles from $E \to X$ to $F \to Y$ is a continuous map $L: E \to F$ that covers some map $u: X \to Y$ (necessarily continuous) and is fiberwise linear. Given $x$ in $X$, we shall be writing $L_x: E_x \to F_{u(x)}$ for the bounded linear map induced between the fiber of $E$ at $x$ and the fiber of $F$ at $u(x)$. A basic example is provided by $u^*F = X \times_Y F \to X$, the \emph{pullback} of $F \to Y$ along the continuous map $u: X \to Y$; the projection $\pr: u^*F \to F$ is an isometric (i.e., norm-preserving) morphism covering $u$. There is a linear correspondence $\Gamma(Y;F) \to \Gamma(X;u^*F)$, $\eta \mapsto u^*\eta$ defined by $\pr \circ u^*\eta = \eta \circ u$. We shall be writing $\BB$ for the category of all Banach bundles, and $\BB(X)$ for the subcategory of all morphisms that cover $u = \id_X$. If $L: F \to F'$ is a morphism of Banach bundles in $\BB(Y)$, then setting $\pr \circ u^*L = L \circ \pr$ yields a morphism $u^*L: u^*F \to u^*F'$ of Banach bundles in $\BB(X)$. In this way we get a functor $\BB(Y) \to \BB(X)$.

Let $E \to X$ be a Banach bundle. Given any continuous cross-section $\zeta \in \Gamma(X;E)$, any open subset $U \subset X$, and any real number $r > 0$, we write \[%
	\mathcal{B}_r(\zeta,U) = \{e \in p^{-1}(U): \lvert e - \zeta(pe)\rvert < r\}.
\] Let $e \in E_x$ be assigned. For each $\zeta \in \Gamma(X;E)$ with $\zeta(x) = e$, the sets $\mathcal{B}_r(\zeta,U)$ with $U \ni x$ form a basis of open neighborhoods of $e$ in $E$ \cite[p.~102]{Fel77}. Let $\lVert\blank\rVert$ be the standard operator norm on the space $\L(\Ban{E},\Ban{F})$ of all bounded linear maps between two Banach spaces $\Ban{E}$, $\Ban{F}$. Although for a general morphism $L: E \to F$ the function $x \mapsto \lVert L_x\rVert$ may not be continuous, the following always holds; cf.~\cite[Lem.~2.3]{Bos11}, \cite[Prop.~2]{Sed82}.

\begin{lem}\label{lem:morphism} Let\/ $L: E \to F$ be any fiberwise linear map between Banach bundles\/ $E \to X$, $F \to Y$ which covers a continuous map\/ $u: X \to Y$. Provided\/ $X$ is either paracompact or locally compact, the following are equivalent.
\begin{enumerate}
\def\labelenumi{\upshape(\roman{enumi})}
 \item $L: E \to F$ is continuous, in other words, it is a Banach bundle morphism.
 \item $L\zeta \in \Gamma(X;u^*F)$ for all\/ $\zeta \in \Gamma(X;E)$, and the function\/ $x \mapsto \lVert L_x\rVert$ (operator norm) is locally bounded on\/ $X$, hence bounded on every compact subset of\/ $X$.
\end{enumerate} \end{lem}

\begin{proof} $L: E \to F$ is continuous iff it is such as a map $E \to u^*F$. We may therefore suppose that $X = Y$ and $u = \id$. Let us assume (i). Since $\mathcal{B}_1(0,X)$ is an open subset of $F$, we can for each $x \in X$ by continuity find $U \ni x$ open and $r > 0$ such that $L$ maps $\mathcal{B}_r(0,U)$ into $\mathcal{B}_1(0,X)$, whence $\lVert L_x\rVert \leq 1/r$ for all $x \in U$. Conversely, let us assume (ii). Let $e \in E_x$ be given. By the result of Douady and dal Soglio-Hérault cited previously, we can find $\zeta \in \Gamma(X;E)$ with $\zeta(x) = e$. The sets $\mathcal{B}_r(L\zeta,U)$ form a basis of open neighborhoods of $Le$ in $F$. Let us pick $b > 0$ such that $\lVert L_u\rVert \leq b$ for all $u \in U$. Then $L$ carries $\mathcal{B}_{r/b}(\zeta,U)$ into $\mathcal{B}_r(L\zeta,U)$ for each $r > 0$. \end{proof}

\begin{lem}\label{lem:limit} Let\/ $L^{(n)}: E \to F$, $n = 0$,~$1$,~$2$,~$\dotsc$, be morphisms of Banach bundles in\/ $\BB(X)$ over a paracompact or locally compact\/ $X$. Suppose that\/ \(%
	\sup_{x\in X}{}\lVert L^{(m)}_x - L^{(n)}_x\rVert < \epsilon
\) for all sufficiently large\/ $m$,\ $n$ for any given\/ $\epsilon > 0$. A new Banach bundle morphism\/ $L = \lim L^{(n)}: E \to F \in \BB(X)$ is obtained upon taking the limit\/ $L_x$ of the operators\/ $L^{(n)}_x$ in\/ $\L(E_x,F_x)$ for each\/ $x$ in\/ $X$. \end{lem}

\begin{proof} $L$ satisfies the condition (ii) of the previous lemma because any uniform limit of continuous cross-sections is itself one such cross-section; cf.~\cite[p.~102, Prop.~10.3]{Fel77}. \end{proof}

\paragraph*{Multipliers and projective representations.} All Banach algebras $\Ban{A}$ in this paper will be unital and satisfy $\lvert ab\rvert \leq \lvert a\rvert\,\lvert b\rvert$, $\lvert 1\rvert = 1$. All (left) Banach $\Ban{A}$-modules $\Ban{E}$ will satisfy $\lvert ae\rvert \leq \lvert a\rvert\,\lvert e\rvert$, $1e = e$. No notational distinction will be made between the elements $a$ of $\Ban{A}$ and the bounded operators, $e \mapsto ae$, they induce in $\Ban{E}$. We shall call a Banach bundle $A \to X$ endowed with a continuously varying family of bilinear maps $\{A_x \times A_x \to A_x$,~$x \in X\}$ and with a continuous cross-section $x \mapsto 1x$ turning the fibers of $A$ into Banach algebras a \emph{Banach algebra bundle}. We shall write $A^\times$ for the open subbundle of $A \to X$ (in fact topological group bundle) whose fiber $A^\times_x$ is the group of all invertible elements of $A_x$. There is an obvious notion of \emph{Banach\/ $A$-module bundle} $E \to X$ for a Banach algebra bundle $A \to X$. We shall say that a morphism $L: E \to F$ in $\BB(X)$ between two Banach $A$-module bundles $E \to X$, $F \to X$ is \emph{$A$-linear}, and write $L \in \BB_A(X)$, if \(%
	L(ae) = aLe
\) for all $(a,e) \in A \times_X E$. When $A \to X$ is the trivial Banach algebra bundle $X \times \Ban{A} \to X$, we shall speak of \emph{Banach\/ $\Ban{A}$-module bundles} and \emph{$\Ban{A}$-linear morphisms} in $\BB_{\Ban{A}}(X)$. When $\Ban{A} = \C$, resp., $\mathbb{H}$, we shall furthermore speak of \emph{complex}, resp., \emph{quaternionic}, Banach bundles or morphisms.

Whenever a locally compact groupoid $G \tto X$ acts (from the left) on a Banach algebra bundle $A \to X$ in a continuous fashion by isometric (that is, norm-preserving) isomorphisms of algebras $A_{sg} \simto A_{tg}$, $a \mapsto ga$, we call $A \to X$ a \emph{Banach algebra\/ $G$-bundle}. Writing $t$ not only for the usual target map $G \to X$ but also for the map $G \ftimes{s}{t} G \to X$ given by $(g,h) \to tg$, we define a \emph{multiplier} for $G \tto X$ with values in the Banach algebra $G$-bundle $A \to X$, shortly, an \emph{$A$-valued multiplier} for $G \tto X$, to be a continuous cross-section $\sigma \in \Gamma(G \ftimes{s}{t} G;t^*A^\times)$ satisfying the following three axioms:
\begin{enumerate}
\def\labelenumi{\upshape(\roman{enumi})}
 \item $\sigma$ is \emph{central}, in other words, $\sigma(g,h)$ lies in the center of $A_{tg}$ for every composable pair of arrows $g,h$;
 \item $\sigma$ is \emph{normal}, that is to say, $\sigma(g,h) = 1 \in A_{tg}$ whenever either $g = 1th$ or $h = 1sg$;
 \item $\sigma$ is a \emph{$2$-cocycle}, meaning the identity below is satisfied for all composable triplets of arrows $(g,h,k) \in G \ftimes{s}{t} G \ftimes{s}{t} G$.
\begin{equation}
	\sigma(g,h)\sigma(gh,k) = g\sigma(h,k)\sigma(g,hk)
\label{eqn:2-cocycle}
\end{equation} \end{enumerate}

Let $E \to X$ be a Banach $A$-module bundle for the Banach algebra $G$-bundle $A \to X$. Let $\sigma$ be an $A$-valued multiplier for $G \tto X$. For every Banach bundle morphism $R: s^*E \to t^*E$ in $\BB(G)$, we write $R(g): E_{sg} \to E_{tg}$ for the bounded linear operator corresponding to an arrow $g \in G$. We call $R$ a \emph{$\sigma$-representation} if it satisfies the condition \(%
	R(g)(ae) = ga \cdot R(g)e
\) for all $g \in G$, $a \in A_{sg}$, and $e \in E_{sg}$, as well as satisfying the two conditions below for all $x \in X$ and $(g,h) \in G \ftimes{s}{t} G$.%
\begin{subequations}
\label{eqn:representation}
\begin{gather}
	R(1x) = \id
\label{eqn:unital}\\%
	\sigma(g,h)R(gh) = R(g)R(h)
\label{eqn:multiplicative}
\end{gather}
\end{subequations}
Because of the invertibility of $\sigma(g,g^{-1}) \in A^\times_{tg}$, a $\sigma$-representation is \emph{invertible} in other words is a (bicontinuous) Banach bundle isomorphism $s^*E \simto t^*E$. By an \emph{$A$-linear representation} we shall mean a $\sigma$-representation for the trivial multiplier $\sigma = 1$. By an \emph{$\Ban{A}$-linear representation} for a Banach algebra $\Ban{A}$ we shall mean an $A$-linear representation for the trivial Banach algebra $G$-bundle $A = X \times \Ban{A} \to X$, $g \cdot (sg,a) = (tg,a)$. When $\Ban{A} = \C$, resp., $\mathbb{H}$, we shall speak of \emph{complex}, resp., \emph{quaternionic}, representations. An \emph{intertwiner} between two $\sigma$-representations $R$ and $S$ which operate respectively in $E$ and $F$ is a morphism $L: E \to F$ in $\BB_A(X)$ that (is both $A$-linear and) satisfies
\begin{equation}
	L_{tg}R(g) = S(g)L_{sg}.
\label{eqn:intertwiner}
\end{equation}
The two $\sigma$-representations are \emph{equivalent}, in symbols, $R \sim S$, if it is possible to intertwine them by means of some ($A$-linear) isomorphism in $\BB_A(X)$.

We have in mind not only complex or quaternionic projective representations but also other types of projective representations, say, with coefficients in infinite-dimensional Banach algebras, such as the algebra of continuous functions on a compact Hausdorff space.

\paragraph*{Local triviality and uniform continuity.} In the special case of a locally compact group $G$ operating in a Banach space $\Ban{E}$ by automorphisms $R(g)$ the notion we have just defined coincides with that of a “strongly continuous” representation. In fact, the following three properties of a group homomorphism $R: G \to \GL(\Ban{E})$ are equivalent: (i)~the map $G \times \Ban{E} \to \Ban{E}$, $(g,e) \mapsto R(g)e$ is continuous; (ii)~the map $G \to \Ban{E}$, $g \mapsto R(g)e$ is continuous for each $e \in \Ban{E}$; (iii)~the “generalized matrix elements” of $R$ viz.~the functions on $G$ given for all $e \in \Ban{E}$, $\lambda \in \Ban{E}'$~(= Banach dual of $\Ban{E}$) by $g \mapsto \lambda\bigl(R(g)e\bigr)$ are continuous. (The implication $\text{(ii)} \Rightarrow \text{(i)}$ is a consequence of the Banach–Steinhaus theorem, also known as the “uniform boundedness principle”; as to the implication $\text{(iii)} \Rightarrow \text{(ii)}$, see \cite[p.~89]{Lyu88}.) Although when the Banach spaces $\Ban{E}$ involved are infinite-dimensional strong continuity tends to be the relevant concept insofar as group representations $G \to \GL(\Ban{E})$ are concerned, there are at least two good reasons for us to be also interested in the more restrictive notion of “uniform continuity”, that is, continuity relative to the operator-norm topology on $\GL(\Ban{E})$. The first reason is that we want our almost representation theorem to encompass as a special case the result of de~la~Harpe and Karoubi \cite[p.~294]{dlHK77}, who deal with uniformly continuous representations. [Cf.~our considerations in section \ref{sec:intro}.] The second reason, which we shall not comment on here, is more serious and has to do with our local extension theorem. [The relevant motivations will be discussed after the statement of Theorem \ref{thm:local} in section \ref{sec:local}.] Our treatment of uniform continuity will be tailored to our needs and will differ somewhat from that given in \cite{Bos11}.

Let $A \to X$ be an arbitrary Banach algebra bundle. By a \emph{local trivialization} of a Banach $A$-module bundle $E \to X$ by a Banach space $\Ban{E}$ over an open set $U \subset X$ we mean an isomorphism \(%
	U \times \Ban{E} \simto E \mathbin| U
\) in $\BB(U)$ under which scalar multiplication by elements of $A$ corresponds to a map \(%
	A \mathbin| U \to \End(\Ban{E})
\) that is continuous for the operator-norm topology on \(%
	\End(\Ban{E}) = \L(\Ban{E},\Ban{E})
\). Whenever for a given $\Ban{E}$ it is possible to find a whole atlas of such local trivializations whose transition functions \(%
	U \to \GL(\Ban{E})
\) are continuous for the operator-norm topology, we say that $E \to X$ is \emph{locally trivial of class\/ $C^0$} and call the equivalence class of the atlas a \emph{$\GL(\Ban{E})$-structure} for $E \to X$. We speak of an \emph{$\GI(\Ban{E})$-structure}, where $\GI(\Ban{E}) \subset \GL(\Ban{E})$ denotes the subgroup of all isometric automorphisms of $\Ban{E}$, if the $\GL(\Ban{E})$-structure contains atlases consisting exclusively of \emph{isometric} i.e.~norm-preserving local trivializations. We refer to $E \to X$ as a \emph{$\GL(\Ban{E})$-structured} resp.~\emph{$\GI(\Ban{E})$-structured} Banach $A$-module bundle whenever $E \to X$ comes equipped with a specific choice of $\GL(\Ban{E})$-structure resp.~$\GI(\Ban{E})$-structure. The pullback $u^*E \to Y$ of any locally trivial Banach $A$-module bundle $E \to X$ of class $C^0$ along any continuous map $u: Y \to X$ is a locally trivial Banach $u^*A$-module bundle of class $C^0$.

\begin{rmk*} Let $E \to X$, $F \to X$ be Banach bundles. Let $\mathcal{C}$ be a collection of morphisms $E \to F$ in $\BB(X)$ for which the following two conditions are satisfied:
\begin{enumerate}
\def\labelenumi{\upshape(\roman{enumi})}
 \item $\mathcal{C}$ is a \emph{linear} subspace of the real vector space $\BB(X)(E,F)$.
 \item The numerical function $x \mapsto \lVert L_x\rVert$ (operator norm) on $X$ is \emph{continuous} for each $L \in \mathcal{C}$.
\end{enumerate}
For each $x$ in $X$, let $\clos{\mathcal{C}}_x$ be the closure of the linear subspace $\mathcal{C}_x = \{L_x: L \in \mathcal{C}\}$ of $\L(E_x,F_x)$ under the operator norm. Then, by \cite[p.~102, Prop.~10.4]{Fel77}, there exists a unique topology on the family of Banach spaces $\clos{\mathcal{C}} = \{\clos{\mathcal{C}}_x\} \to X$ turning the latter into a Banach bundle over $X$ in such a way that $\mathcal{C}$ becomes a linear subspace of $\Gamma(X;\clos{\mathcal{C}})$. \end{rmk*}

Now let $E \to X$ and $F \to X$ be respectively an $\GI(\Ban{E})$-\ and an $\GI(\Ban{F})$-structured locally trivial Banach $A$-module bundle of class $C^0$. We shall say that an $A$-linear morphism $L: E \to F \in \BB(X)$ is \emph{uniformly continuous} if, for each pair of local trivializations \(%
	U \times \Ban{E} \simto E \mathbin| U$ and $%
	U \times \Ban{F} \simto F \mathbin| U
\) which are of class $C^0$ in relation to the given $\GL$-structures, the “local matrix expression” of $L$ is a map $U \to \L(\Ban{E},\Ban{F})$ continuous for the operator-norm topology on $\L(\Ban{E},\Ban{F})$. The collection $\mathcal{C} \subset \BB(X)(E,F)$ of all such morphisms evidently satisfies the conditions (i)–(ii) of our previous remark; let us write $\L_A(E,F) \to X$ for the corresponding “internal-hom” Banach bundle $\clos{\mathcal{C}} \to X$. The topology of this Banach bundle does a~priori depend on the choice of $\GL$-structures but its norm function is invariably given by the operator norm in each fiber. It is not hard to see that its continuous cross-sections happen to be precisely the uniformly continuous $A$-linear morphisms $E \to F$. Since in general we only have $\L_A(E,F)_x \subset \L_{A_x}(E_x,F_x)$, without equality, $\L_A(E,F) \to X$ won't be itself locally trivial unless further trivializability hypotheses are made on $A \to X$ and on its way of operating on $E \to X$, $F \to X$.

Let $G \tto X$ be a locally compact groupoid and let $A \to X$ be a Banach algebra $G$-bundle. Put $B = s^*A \to G$, viewed as a Banach algebra bundle, and, for any Banach $A$-module bundle $E \to X$, let both $s^*E \to G$ and $t^*E \to G$ be regarded as Banach $B$-module bundles; in the case of $t^*E \to G$, the scalar multiplication is the one carried over along the isometric Banach algebra isomorphism $B = s^*A \simto t^*A$ (change of coefficients) provided by the $G$~action. Let $\sigma$ be a multiplier for $G \tto X$ with values in $A \to X$. Assuming that $E \to X$ is an $\GI(\Ban{E})$-structured locally trivial Banach $A$-module bundle of class $C^0$, a $\sigma$-representation of $G \tto X$ on $E \to X$ is \emph{uniformly continuous} if it is so as a morphism $s^*E \to t^*E$ in $\BB_B(G)$, equivalently, if it is continuous as a cross-section of the Banach bundle $\L_B(s^*E,t^*E) \to G$, each one of $s^*E \to G$, $t^*E \to G$ being endowed with its canonical pullback $\GI(\Ban{E})$-structure.

\section{The almost representation theorem}\label{sec:almost}

Let us now fix a locally compact groupoid $G \tto X$, a Banach algebra $G$-bundle $A \to X$, a Banach $A$-module bundle $E \to X$, and an $A$-valued multiplier $\sigma$ for $G \tto X$. By a \emph{pseudorepresentation} of $G \tto X$ on $E \to X$ we mean an arbitrary Banach bundle morphism $T: s^*E \to t^*E$ in $\BB(G)$. If the condition \(%
	T(g)(ae) = ga \cdot T(g)e
\) is satisfied for all $g \in G$, $a \in A_{sg}$, and $e \in E_{sg}$, we say that $T$ is \emph{$A$-linear}.

\begin{defn}\label{defn:almost} The \emph{$\sigma$-defect} of an arbitrary $A$-linear pseudorepresentation $T: s^*E \to t^*E$ along an orbit $Gx$ is the nonnegative (possibly infinite) quantity
\begin{equation}
	r(T,x) = \sup_{y\in Gx}{}\lVert\id - T(1y)\rVert
	         + \sup_{\substack{g\in t^{-1}(Gx)\\ h\in G^{sg}}} \ell(\sigma,g)\lVert\sigma(g,h)T(gh) - T(g)T(h)\rVert,
\label{eqn:defect}
\end{equation}
where $\ell(\sigma,g) := \max\{1,\lvert\sigma(g,g^{-1})^{-1}\rvert\}$. The \emph{$\sigma$-bound} of $T$ along $Gx$ is the quantity
\begin{equation}
	b(T,x) = \sup_{g\in t^{-1}(Gx)} \ell(\sigma,g)\lVert T(g)\rVert.
\label{eqn:bound}
\end{equation}
We say that $T: s^*E \to t^*E$ is an \emph{almost\/ $\sigma$-representation} if for every $x$ the two quantities $b(T,x)$ and $r(T,x)$ are related by the inequality
\begin{equation}
	r(T,x) \leq \min\{1/4,b(T,x)^{-2}/9\}.
\label{eqn:almost}
\end{equation} \end{defn}

By a \emph{cutoff function} for $G \tto X$ we mean a nonnegative continuous function $c: X \to \R_{\geq 0}$ with the following couple of properties: (i)~$\supp(c \circ s) \cap t^{-1}(C)$ is a compact set in $G$ for every compact set $C$ in $X$; (ii)~every orbit meets the open set where $c > 0$. We say that $G \tto X$ is a \emph{proper} locally compact groupoid, or briefly a \emph{proper groupoid}, if it admits cutoff functions. The existence of such functions implies that the groupoid anchor $(t,s): G \to X \times X$ is a proper map (all inverse images of compact sets are compact), but the properness of $(t,s)$ implies the existence of cutoff functions only provided $X$ is second countable \cite[§6.2]{Tu99}; on the other hand, cutoff functions may well exist in cases where $X$ is not second countable---they do, for instance, when $G \tto X$ is a locally compact group bundle with compact isotropy groups, see Example (c) at the beginning of section \ref{sec:representations}, whether or not $X$ is second countable.

\begin{thm}\label{thm:almost} Let\/ $T: s^*E \to t^*E$ be an almost\/ $\sigma$-representation of\/ $G \tto X$ on\/ $E \to X$. Suppose that\/ $G \tto X$ is proper. Then, there exists a\/ $\sigma$-representation\/ $R: s^*E \simto t^*E$ lying within a distance of\/ $4b(T)r(T)$ from\/ $T$ in the sense that\/ $\lVert R(g) - T(g)\rVert \leq 4b(T,tg)r(T,tg)$ for all\/ $g$ in\/ $G$. If\/ $E \to X$ is locally trivial of class\/ $C^0$, and if\/ $T$ is uniformly continuous with respect to a given\/ $\GI(\Ban{E})$-structure for\/ $E \to X$, then in addition such an\/ $R$ can be taken to be uniformly continuous relative to the same\/ $\GI(\Ban{E})$-structure. \end{thm}

\subsection*{Proof of the almost representation theorem}

\begin{lem}\label{lem:inverse} Any almost\/ $\sigma$-representation\/ $T: s^*E \to t^*E$ of a locally compact groupoid\/ $G \tto X$ is an invertible morphism in the category\/ $\BB(G)$. Its inverse, $T^{-1}: t^*E \simto s^*E \in \BB(G)$, satisfies the following inequality for all\/ $g \in G$.
\begin{equation}
	\lVert T(g)^{-1}\rVert \leq \frac{b(T,tg)}{1 - r(T,tg)}
\label{eqn:inverse}
\end{equation} \end{lem}

\begin{proof} Let us set $L(g) = \id - \sigma(g,g^{-1})^{-1}T(g)T(g^{-1})$. We have
\begin{multline*}
 \lVert L(g)\rVert \leq \lVert\id - T(1tg)\rVert + \lvert\sigma(g,g^{-1})^{-1}\rvert\:\lVert\sigma(g,g^{-1})T(gg^{-1}) - T(g)T(g^{-1})\rVert \leq {}\\*
	r(T,tg) \leq 1/4 < 1.
\end{multline*}
In view of Lemma \ref{lem:limit} as applied to the sequence of partial sums of a series of endomorphisms of the Banach bundle $t^*E \to G$, $\id - L$ is an automorphism of $t^*E$ and its inverse is given by
\begin{equation*}
	(\id - L)^{-1} = \id + L + L^2 + L^3 + \dotsb \in \BB(G).
\end{equation*}
Thus, both $\sigma(g,g^{-1})^{-1}T(g)T(g^{-1})$ and, hence, $T(g)T(g^{-1})$ are invertible for every $g$. Substituting $g^{-1}$ for $g$, we have that $T(g^{-1})T(g)$ is also invertible and, consequently, so is $T(g)$.

Let us now turn our attention to the estimate \eqref{eqn:inverse}. Let us write $r = r(T,tg)$. We have
\begin{multline*}
 \lVert[\sigma(g,g^{-1})^{-1}T(g)T(g^{-1})]^{-1} - \id\rVert = \lVert[\id - L(g)]^{-1} - \id\rVert \leq {}\\*
	\lVert L(g)\rVert + \lVert L(g)\rVert^2 + \lVert L(g)\rVert^3 + \dotsb \leq r + r^2 + r^3 + \dotsb = r(1 - r)^{-1}.
\end{multline*}
By the cocycle identity \eqref{eqn:2-cocycle}, $\sigma(g,g^{-1}) = g\sigma(g^{-1},g)$. Thus,
\begin{alignat*}{2} &
 \lVert\sigma(g^{-1},g)T(g)^{-1} - T(g^{-1})\rVert \\ &\hskip+3em%
	\leq \lVert T(g^{-1})\rVert\:\lVert T(g^{-1})^{-1}\sigma(g^{-1},g)T(g)^{-1} - \id\rVert \\ &\hskip+3em%
	= \lVert T(g^{-1})\rVert\:\lVert[T(g)\sigma(g^{-1},g)^{-1}T(g^{-1})]^{-1} - \id\rVert \\ &\hskip+3em%
	= \lVert T(g^{-1})\rVert\:\lVert[g\sigma(g^{-1},g)^{-1}T(g)T(g^{-1})]^{-1} - \id\rVert &\quad &\text{%
			by $A$-linearity of $T$} \\ &\hskip+3em%
	= \lVert T(g^{-1})\rVert\:\lVert[\sigma(g,g^{-1})^{-1}T(g)T(g^{-1})]^{-1} - \id\rVert \\ &\hskip+3em%
	\leq r(1 - r)^{-1}\lVert T(g^{-1})\rVert.
\end{alignat*}
It follows from the latter inequality that
\begin{align*}
 \lVert T(g)^{-1}\rVert &
	\leq \lvert\sigma(g^{-1},g)^{-1}\rvert\:\lVert\sigma(g^{-1},g)T(g)^{-1}\rVert \\ &
	\leq \lvert\sigma(g^{-1},g)^{-1}\rvert\{\lVert T(g^{-1})\rVert + \lVert\sigma(g^{-1},g)T(g)^{-1} - T(g^{-1})\rVert\} \\ &
	\leq \max\{1,\lvert\sigma(g^{-1},g)^{-1}\rvert\}\lVert T(g^{-1})\rVert\left(1 + \frac{r}{1 - r}\right) \\ &
	\leq \frac{b(T,tg)}{1 - r(T,tg)}.
\qedhere
\end{align*} \end{proof}

For any given Haar measure system $\{\mu^x$,~$x \in X\}$ on the proper groupoid $G \tto X$ it is possible to construct a \emph{normalizing function}, that is a cutoff function $c$ such that $\integral_{th=x} c(sh) \der h = 1$ for all $x$, simply by dividing an arbitrary cutoff function $c$ by the positive function $x \mapsto \integral_{th=x} c(sh) \der h$. We shall regard $\{\mu^x\}$ and $c$ as fixed throughout what follows.

Let $q: Y \to X$ be an arbitrary continuous map of another locally compact Hausdorff space $Y$ into $X$. The fiber product $Y \ftimes{q}{t} G = \{(y,h): q(y) = th\}$ is a closed subspace of $Y \times G$, hence itself a locally compact space, with continuous projection $\pr: Y \ftimes{q}{t} G \to Y$. Let $F \to Y$ be an arbitrary Banach bundle over $Y$. For any continuous cross-section $\vartheta \in \Gamma(Y \ftimes{q}{t} G;\pr^*F)$, the expression $\integral c(sh)\vartheta(y,h) \der h$ makes sense for each $y \in Y$ as an $F_y$~valued integral over $G^{q(y)}$ because the integrand is an $F_y$~valued continuous function vanishing outside a compact set. We claim that $y \mapsto \integral c(sh)\vartheta(y,h) \der h \in F_y$ is in fact a \emph{continuous} cross-section of $F \to Y$ and that consequently there is a linear correspondence
\begin{equation}
	\Gamma(Y \ftimes{q}{t} G;\pr^*F) \longto \Gamma(Y;F), \quad\textstyle
		\vartheta \longmapsto \integral c(sh)\vartheta(-,h) \der h.
\label{eqn:Haar}
\end{equation}
To see it, we observe the following: (i)~trivially, our claim is true for any $\vartheta$ of the form $\pr^*\eta$, where $\eta \in \Gamma(Y;F)$; (ii)~it is also true when $\vartheta$ is a finite linear combination $\sum\varphi_i\pr^*\eta_i$ with coefficients $\varphi_i$ in $C_c(Y \ftimes{q}{t} G)$, cf.~the argument in the proof of \cite[p.~48, Prop.~1.1]{Ren80}; (iii)~an arbitrary $\vartheta$ can be approximated as closely as we wish in norm on any given compact subset of $Y \ftimes{q}{t} G$ by some such linear combination, cf.~\cite[p.~104, Prop.~10.6]{Fel77}; (iv)~the existence of arbitrarily close locally approximating continuous cross-sections is equivalent to continuity \cite[p.~102, Prop.~10.3]{Fel77}.

Now, if $T: s^*E \simto t^*E \in \BB(G)$ is any invertible $A$-linear pseudorepresentation of $G \tto X$ on the Banach $A$-module bundle $E \to X$, it makes sense for all $g \in G$, $e \in E_{sg}$ to set
\begin{equation}
	\avg{T}(g)e = \integral_{th=sg} c(sh)\sigma(g,h)T(gh)T(h)^{-1}e \der h \in E_{tg}
\label{eqn:mean}
\end{equation}
(the integrand being a compactly supported continuous $E_{tg}$~valued function on $G^{sg}$). The fiberwise linear map $\avg{T}: s^*E \to t^*E$ provided by equation \eqref{eqn:mean} covers the identical transformation of $G$. We contend it is a Banach bundle morphism in $\BB(G)$ and hence a pseudorepresentation of $G \tto X$ on $E \to X$. First, since by assumption $T$ has an inverse $T^{-1}: t^*E \simto s^*E$ in $\BB(G)$, by Lemma \ref{lem:morphism} the function on $G \ftimes{s}{t} G$ given by $g,h \mapsto \lVert\sigma(g,h)T(gh)T(h)^{-1}\rVert$ is locally bounded and hence bounded on each compact set in particular on \(%
	K \ftimes{s}{t} \bigl[\supp(c \circ s) \cap t^{-1}\bigl(s(K)\bigr)\bigr]
\) for each compact $K$ in $G$. Thus, the function $g \mapsto \lVert\avg{T}(g)\rVert$ (exists and) is locally bounded. Second, given $\eta \in \Gamma(G;s^*E)$, upon applying our Haar integration functional \eqref{eqn:Haar} with $Y = G \xto{s} X$ and $F = t^*E$ to $\vartheta(g,h) = \sigma(g,h)T(gh)T(h)^{-1}\eta(g)$, we see that $g \mapsto \avg{T}(g)\eta(g)$ lies in $\Gamma(G;t^*E)$, whence our contention is proven on account of Lemma \ref{lem:morphism}. Because of the centrality of $\sigma$, our pseudorepresentation $\avg{T}$ must be $A$-linear. In view of the normality of $\sigma$, it must further satisfy $\avg{T}(1x) = \id$ for all $x$ in $X$ and so be \emph{unital}. When $T$ is a $\sigma$-representation, our normalization hypothesis $\integral_{th=x} c(sh) \der h = 1$ tells us that---suppressing $e$ henceforth systematically \[%
\textstyle
	T(g) = \integral c(sh)T(g) \der h
	     = \integral c(sh)T(g)T(h)T(h)^{-1} \der h
	     = \integral c(sh)\sigma(g,h)T(gh)T(h)^{-1} \der h
\] and therefore $\avg{T} = T$: any $\sigma$-representation goes unaffected through our operator $T \mapsto \avg{T}$. On the other hand, when it is simply an almost $\sigma$-representation, $T$ must be invertible by Lemma \ref{lem:inverse} and hence $\avg{T}$ must be defined.

\begin{lem}\label{lem:estimates} For any almost\/ $\sigma$-representation\/ $T: s^*E \simto t^*E$ the following inequalities hold for all\/ $x$ in\/ $X$.
\begin{gather}
	b(\avg{T},x) \leq \frac{b(T,x)}{1 - r(T,x)}
\label{eqn:estimate1}\\
	r(\avg{T},x) \leq 2\left(\frac{b(T,x)}{1 - r(T,x)}\right)^2r(T,x)^2
\label{eqn:estimate2}
\end{gather} \end{lem}

\begin{proof} Given $(g,h) \in G \ftimes{s}{t} G$ with $tg \in Gx$, by virtue of our estimate \eqref{eqn:inverse} for the inverse of $T$, we have
\begin{align} &
 \lVert\sigma(g,h)T(gh)T(h)^{-1} - T(g)\rVert
	\leq \lVert\sigma(g,h)T(gh) - T(g)T(h)\rVert\:\lVert T(h)^{-1}\rVert \notag\\ &\hskip+3em
	\leq \max\{1,\lvert\sigma(g,g^{-1})^{-1}\rvert\}^{-1}r(T,x)\frac{b(T,x)}{1 - r(T,x)}.
\label{eqn:12B.12.9b}
\end{align}

By the normalizing function property $\integral c(sh) \der h = 1$ (still suppressing $e$ systematically)
\begin{align}
 \avg{T}(g) - T(g)
	&= \integral c(sh)\sigma(g,h)T(gh)T(h)^{-1} \der h - \integral c(sh)T(g) \der h \notag\\
	&= \integral c(sh)[\sigma(g,h)T(gh)T(h)^{-1} - T(g)] \der h
\label{eqn:12B.12.5a}
\end{align}
and therefore setting $r = r(T,x)$
\begin{align*} &
 \max\{1,\lvert\sigma(g,g^{-1})^{-1}\rvert\} \lVert\avg{T}(g)\rVert \\ &\hskip+3em
	\leq \max\{1,\lvert\sigma(g,g^{-1})^{-1}\rvert\}\lVert T(g)\rVert
			+ \max\{1,\lvert\sigma(g,g^{-1})^{-1}\rvert\}\lVert\avg{T}(g) - T(g)\rVert \\ &\hskip+3em
	\leq b(T,x) + b(T,x)r(1 - r)^{-1} \\ &\hskip+3em
	= \left(1 + \frac{r}{1 - r}\right)b(T,x).
\end{align*}
We obtain our first estimate \eqref{eqn:estimate1} by taking the sup over all $g$.

Next, because of the Haar system's left invariance and, again, of the normalizing function property,
\begin{alignat*}{2} &
 \sigma(g_1,g_2)\avg{T}(g_1g_2) - \avg{T}(g_1)\avg{T}(g_2)
	= \sigma(g_1,g_2)\avg{T}(g_1g_2) - {}\\* & & &\llap{$\displaystyle%
				\left(\integral c(sh)\sigma(g_1,h)T(g_1h)T(h)^{-1} \der h\right) \circ \avg{T}(g_2)$} \\ &
	= \sigma(g_1,g_2)\avg{T}(g_1g_2)
		- \integral c(sh')\sigma(g_1,g_2h')T(g_1g_2h')T(g_2h')^{-1} \circ \avg{T}(g_2) \der h' \\ &
	= \integral c(sh)\sigma(g_1,g_2)\sigma(g_1g_2,h)T(g_1g_2h)T(h)^{-1} \der h \\* &\justify
		- \iintegral c(sh)c(sh')\sigma(g_1,g_2h')T(g_1g_2h')T(g_2h')^{-1} \circ \sigma(g_2,h)T(g_2h)T(h)^{-1} \der h \der h' \\%
\intertext{%
and therefore, on account of the $2$-cocycle identity $\sigma(g_1,g_2)\sigma(g_1g_2,h) = g_1\sigma(g_2,h)\sigma(g_1,g_2h)$, of the centrality of $\sigma$, and of the $A$-linearity of $T$,} &
	= \integral c(sh)\sigma(g_1,g_2h)T(g_1g_2h)T(g_2h)^{-1} \circ \sigma(g_2,h)T(g_2h)T(h)^{-1} \der h \\* &\hskip+5em
			- \integral c(sh)\sigma(g_1,g_2h)T(g_1g_2h)T(g_2h)^{-1} \circ T(g_2) \der h \\* & & &\llap{${}\displaystyle%
				- \integral c(sh)T(g_1) \circ \sigma(g_2,h)T(g_2h)T(h)^{-1} \der h
				+ T(g_1)T(g_2)$} \\* &\hskip+5em
			+ \integral c(sh)\sigma(g_1,g_2h)T(g_1g_2h)T(g_2h)^{-1} \circ T(g_2) \der h \\* & & &\llap{${}\displaystyle%
				+ \integral c(sk)T(g_1) \circ \sigma(g_2,k)T(g_2k)T(k)^{-1} \der k
				- T(g_1)T(g_2)$} \\* &\justify
		- \iintegral c(sh)c(sk)\sigma(g_1,g_2h)T(g_1g_2h)T(g_2h)^{-1} \circ \sigma(g_2,k)T(g_2k)T(k)^{-1} \der h \der k \\ &
	= \integral c(sh)[\sigma(g_1,g_2h)T(g_1g_2h)T(g_2h)^{-1} - T(g_1)] \circ [\sigma(g_2,h)T(g_2h)T(h)^{-1} - T(g_2)] \der h \\* &\justify
		- \iint c(sh)c(sk)[\sigma(g_1,g_2h)T(g_1g_2h)T(g_2h)^{-1} - T(g_1)] \circ {}\\*[-.3\baselineskip] & & &\llap{$%
					[\sigma(g_2,k)T(g_2k)T(k)^{-1} - T(g_2)] \thinspace\der h \thinspace\der k$}.
\end{alignat*}
Consequently
\begin{multline*}
 \max\{1,\lvert\sigma(g_1,g_1^{-1})^{-1}\rvert\}\lVert\sigma(g_1,g_2)\avg{T}(g_1g_2) - \avg{T}(g_1)\avg{T}(g_2)\rVert \leq {}\\
	2\sup_{th=sg_2} \max\{1,\lvert\sigma(g_1,g_1^{-1})^{-1}\rvert\}\lVert\sigma(g_1,g_2h)T(g_1g_2h)T(g_2h)^{-1} - T(g_1)\rVert \times {}\\*[-.4\baselineskip]
		\sup_{tk=sg_2}{}\lVert\sigma(g_2,k)T(g_2k)T(k)^{-1} - T(g_2)\rVert.
\end{multline*}
Since $\avg{T}$ is unital (because of the normality of $\sigma$), our second estimate \eqref{eqn:estimate2} follows from \eqref{eqn:12B.12.9b} once we take notice of the fact that $\max\{1,\lvert\sigma(g_2,g_2^{-1})^{-1}\rvert\}^{-1} \leq 1$. \end{proof}

Our last lemma ensures that $\avg{T}$ itself must be an almost $\sigma$-representation: let us write $b_0 = b(T,x)$, $r_0 = r(T,x)$, $b_1 = b(\avg{T},x)$, and $r_1 = r(\avg{T},x)$; by \eqref{eqn:estimate1} and \eqref{eqn:almost}, $b_1 \leq b_0/(1 - r_0) \leq \frac{4}{3}b_0$; by \eqref{eqn:estimate2}, we have $r_1 \leq 2[b_0/(1 - r_0)]^2r_0^2 \leq 2\frac{16}{9}b_0^2r_0^2$, which in combination with \eqref{eqn:almost} implies that $r_1 \leq 2\frac{16}{9}\frac{1}{9}r_0 \leq \frac{1}{2}r_0 \leq 1/8 < 1/4$ and \(%
 b_1^2r_1
	\leq 2(\tfrac{16}{9} \cdot b_0^2r_0)^2
	\leq 2(2 \cdot \tfrac{1}{9})^2
	< 1/9
\); this proves our claim about $\avg{T}$, as well as proving the two inequalities $b_1 \leq \frac{4}{3}b_0$, $r_1 \leq \frac{1}{2}r_0$.

Our almost $\sigma$-representation, $T$, gives rise to a whole sequence, $\avg{T}^0$,~$\avg{T}^1$,~$\avg{T}^2$,~$\dotsc$, of \emph{averaging iterates}: \(%
	\avg{T}^0 = T\), \(%
	\avg{T}^{i+1} = \append[T]\avg{\avg{T}^i}
\) for all $i$. Let us write $b_i = b(\avg{T}^i,x)$, $r_i = r(\avg{T}^i,x)$. By our previous observations, $b_i \leq (\frac{4}{3})^ib_0$, $r_i \leq (\frac{1}{2})^ir_0$, whence for all $g$ in $G^x$, by \eqref{eqn:12B.12.5a} and \eqref{eqn:12B.12.9b} in combination with the inequality $\max\{1,\lvert\sigma(g,g^{-1})^{-1}\rvert\}^{-1} \leq 1$, and on account of \eqref{eqn:almost},
\begin{multline*}
 \lVert\avg{T}^{i+1}(g) - \avg{T}^i(g)\rVert \leq \integral_{th=sg} c(sh)\lVert\sigma(g,h)\avg{T}^i(gh)\avg{T}^i(h)^{-1} - \avg{T}^i(g)\rVert \der h \leq {}\\*
	\frac{1}{1 - r_i} \cdot b_ir_i \leq \frac{1}{1 - r_0} \cdot (\tfrac{4}{3})^i(\tfrac{1}{2})^ib_0r_0 \leq (\tfrac{2}{3})^i\tfrac{1}{3}.
\end{multline*}
Lemma \ref{lem:limit} then implies that the sequence $\avg{T}^0$,~$\avg{T}^1$,~$\avg{T}^2$,~$\dotsc$ converges uniformly to a limiting morphism $\avg{T}^\infty: s^*E \to t^*E$ in $\BB(G)$. Since $\avg{T}^\infty(g) = \lim \avg{T}^i(g)$ in the Banach space $\L(E_{sg},E_{tg})$ for every $g$ and since $\avg{T}^i$ is unital for every $i \geq 1$, we have $\avg{T}^\infty(1x) = \lim\avg{T}^i(1x) = \id$ for all $x$ in $X$. Moreover $\sigma(g,h)\avg{T}^\infty(gh) = \avg{T}^\infty(g)\avg{T}^\infty(h)$ for every composable pair $(g,h) \in G \ftimes{s}{t} G$, because, by continuity,
\begin{multline*}
 \max\{1,\lvert\sigma(g,g^{-1})^{-1}\rvert\}\lVert\sigma(g,h)\avg{T}^\infty(gh) - \avg{T}^\infty(g)\avg{T}^\infty(h)\rVert = {}\\*
	\lim_i\max\{1,\lvert\sigma(g,g^{-1})^{-1}\rvert\}\lVert\sigma(g,h)\avg{T}^i(gh) - \avg{T}^i(g)\avg{T}^i(h)\rVert \leq \lim r_i \leq \lim 2^{-i}r_0 = 0.
\end{multline*}
Thus $R = \avg{T}^\infty$ is a $\sigma$-representation, and the first claim in our theorem is proven.

Next, writing $B = s^*A \to G$, whenever $T \in \Gamma\bigl(G;\L_B(s^*E,t^*E)\bigr)$ is uniformly continuous we may interpret our averaging formula \eqref{eqn:mean} as an instance of the Haar integration functional \eqref{eqn:Haar} with “parameters” in $Y = G \xto{s} X$ and with “coefficients” in $\L_B(s^*E,t^*E) \to G$ and hence as defining a new continuous cross-section $\avg{T}$ of $\L_B(s^*E,t^*E) \to G$. So, in this case, $\avg{T}$ is itself uniformly continuous, and the sequence of averaging iterates of $T$ is one of continuous cross-sections of $\L_B(s^*E,t^*E) \to G$ converging uniformly in norm. On account of \cite[p.~102, Prop.~10.3]{Fel77}, the limiting cross-section of any such sequence must be continuous: the second claim in our theorem is proven too.

\section{Inverse and direct images of unitary representations}\label{sec:images}

In this section we work out what might rightfully be called “Morita invariance theory” for projective representations of locally compact groupoids on Banach bundles, in a form suitable for the applications to be discussed in the next two sections. It turns out that the theory in question does not run so smoothly unless we limit its scope to representations operating by isometric or unitary linear transformations, so let us begin with a quick review of the latter.

\paragraph*{Isometric and unitary representations.} Let $R: s^*E \simto t^*E$ be a $\sigma$-representation of the locally compact groupoid $G \tto X$ on a Banach $A$-module bundle $E \to X$, where $\sigma$ is a multiplier for $G \tto X$ with values in a Banach algebra $G$-bundle $A \to X$. Provided $\sigma$ is \emph{isometric} i.e.~satisfies the condition \(%
	\lvert\sigma(g,h)\rvert = \lvert\sigma(g,h)^{-1}\rvert = 1
\), we refer to $R$ as an \emph{isometric} $\sigma$-representation, or as a $\sigma$-representation \emph{by isometries}, whenever the operators $R(g): E_{sg} \simto E_{tg}$ are norm preserving in the sense that \(%
	\lvert R(g)e\rvert = \lvert e\rvert
\). We say that two such $\sigma$-representations $R$,\ $S$ are \emph{isometrically equivalent}, and write $R \simeq S$, whenever they are intertwined by some ($A$-linear) isomorphism in $\BB_A(X)$ which is itself norm preserving.

Group representations on Banach spaces do not normally behave very well \cite[§8]{Mac63} and in order to set up a workable theory it is convenient to confine attention to unitary representations on Hilbert spaces. The same idea carries over into the present more general context.

Recall that a \emph{Banach\/ $\ast$-algebra} is a Banach algebra $\Ban{A}$ on which some isometric real-linear involution $a \mapsto a^*$ has been defined satisfying $1^* = 1$ and $(ab)^* = b^*a^*$. The basic examples to keep in mind are $\Ban{A} = \R$,~$\C$,~and $\mathbb{H}$; in the real case, $a^* = a$, in the complex case, $(a + ib)^* = a - ib$, and in the quaternionic case, $(a + ib + jc + kd)^* = a - ib -jc -kd$. There is an evident corresponding notion of \emph{Banach\/ $\ast$-algebra bundle}, $A \to X$. A \emph{Hilbert\/ $\Ban{A}$-module} for a Banach $\ast$-algebra $\Ban{A}$ is a Banach $\Ban{A}$-module $\Ban{E}$ equipped with an \emph{$\Ban{A}$-valued inner product} i.e.~a real-bilinear map $\langle\blank,\blank\rangle: \Ban{E} \times \Ban{E} \to \Ban{A}$ satisfying \(%
	\langle ae,f\rangle = a\langle e,f\rangle\), \(%
	\langle e,f\rangle^* = \langle f,e\rangle\), \(%
	\lvert e\rvert^2 = \lvert\langle e,e\rangle\lvert\), and \(%
	\langle e,e\rangle \geq 0
\) (meaning $\langle e,e\rangle$ can be written in the form $a^*a$) for all $a \in \Ban{A}$, $e$,~$f \in \Ban{E}$ \cite{Fel77,Kum98}. Let $E \to X$ be a Banach $A$-module bundle for a Banach $\ast$-algebra bundle $A \to X$. We call $E \to X$ a \emph{Hilbert\/ $A$-module bundle} if its norm function is given via the assignment of a continuously varying family of inner products $\langle\blank,\blank\rangle_x: E_x \times E_x \to A_x$. We shall write $\HB_A(X)$ for the category of all Hilbert $A$-module bundles over $X$.

Let $R: s^*E \simto t^*E$ be a $\sigma$-representation of $G \tto X$ on a Hilbert $A$-module bundle $E \to X$, where $A \to X$ is a Banach $\ast$-algebra $G$-bundle and $\sigma$ is an $A$-valued multiplier for $G \tto X$ which is \emph{unitary} in the sense that \(%
	\sigma(g,h)^{-1} = \sigma(g,h)^*
\). We call $R$ a \emph{unitary} $\sigma$-representation if every operator $R(g): E_{sg} \simto E_{tg}$ preserves the inner product in other words satisfies the condition \(%
	\langle R(g)e,R(g)f\rangle_{tg} = g \cdot \langle e,f\rangle_{sg}
\). We say that two such representations $R$,\ $S$ are \emph{unitarily equivalent}, in symbols, $R \simeq S$, if they are intertwined by some isometric equivalence that also preserves the inner product. It is not hard to see that in the real, complex, or quaternionic case the notion of an isometric $\sigma$-representation or equivalence coincides with that of a unitary $\sigma$-representation or equivalence; to some extent, this mitigates the potential ambiguity caused by our use of the same symbol $\simeq$ to express either isometric or unitary equivalence; in other situations, the intended interpretation may vary depending on the context.

\paragraph*{The general setup.} We shall be dealing with the following situation\spacefactor3000: (i)~$G \tto X$,\ $H \tto Y$ are locally compact groupoids and $\phi: H \to G$ is a continuous homomorphism covering a map $Y \to X$ which we shall also denote by $\phi$. (ii)~$A \to X$ is a Banach algebra $G$-bundle, $B \to Y$ is a Banach algebra $H$-bundle, and $\chi: B \to A$ is a homomorphism of Banach algebra bundles that covers $\phi: Y \to X$ and is both \emph{isometric}, \(%
	\lvert\chi(b)\rvert = \lvert b\rvert
\), and \emph{$\phi$-equivariant}, \(%
	\chi(hb) = \phi(h)\chi(b)
\). (iii)~$\sigma$ is an isometric $A$-valued multiplier for $G \tto X$, $\tau$ is an isometric $B$-valued multiplier for $H \tto Y$, and the two are \emph{compatible} via $\phi$,\ $\chi$ in the sense that \(%
	\chi\bigl(\tau(h,k)\bigr) = \sigma\bigl(\phi(h),\phi(k)\bigr)
\). We shall be interested in the operation of pulling back $\sigma$-representations of $G \tto X$ to $\tau$-representations of $H \tto Y$ and, more importantly, in that of pushing forward $\tau$-representations to $\sigma$-representations.

We start with the easier, and more natural, direction. Let $R: s^*E \simto t^*E$ be an arbitrary $\sigma$-representation of $G \tto X$ on a Banach $A$-module bundle $E \to X$. The pullback Banach bundle, $F = \phi^*E \to Y$, becomes a Banach $B$-module bundle under the operation \(%
	b(y,e) = \bigl(y,\chi(b)e\bigr)
\). (Note that we already need the fact that $\chi$ is isometric here in order to keep in line with our overall conventions about Banach module bundles.) By definition the \emph{inverse image}, or \emph{pullback}, of $R$ along $\phi$ is given by \[%
	S: s^*F \simeq \phi^*s^*E \xto{\phi^*R} \phi^*t^*E \simeq t^*F \text{,\quad equivalently,}\quad
	S(h)(sh,e) = \bigl(th,R(\phi h)e\bigr).
\] It is clear that the morphism $S$ in $\BB(H)$ thus obtained is indeed a $\tau$-representation of $H \tto Y$ on the Banach $B$-module bundle $F \to Y$: it is $B$-linear as a consequence of the $\phi$-equivariance of $\chi$, and it satisfies \eqref{eqn:multiplicative} because of our hypothesis about how $\sigma$ and $\tau$ relate. We shall normally write $\phi^*R$ for the inverse image $\tau$-representation $S$. Its construction extends to intertwiners of $\sigma$-representations and therefore gives rise to a canonical functor from (isometric) $\sigma$-\ to (isometric) $\tau$-representations.

Let us now consider the other direction. Ideally, we would like to assign each $\tau$-representation $S$ of $H \tto Y$ a corresponding \emph{direct image}, or \emph{pushforward}, $\sigma$-representation $\phi_*S$ of $G \tto X$ canonically. We would also like the operation $S \mapsto \phi_*S$ to behave functorially and moreover so that $\phi_*\phi^*R \simeq R$, $\phi^*\phi_*S \simeq S$ naturally in $R$,\ $S$. Unfortunately is is unclear how to do this unless we ask that $S$ be an \emph{isometric} $\tau$-representation, as well as requiring that $\phi$,\ $\chi$ enjoy the following three properties:
\begin{enumerate}
\def\labelenumi{\upshape PF\arabic{enumi}.}
 \item Given any continuous maps $g: U \to G$, $y$,~$z: U \to Y$ such that $sg = \phi y$ and $tg = \phi z$, there is a unique continuous map $h: U \to H$ such that $sh = y$, $th = z$, and $\phi h = g$.
 \item The map $Z := G \ftimes{s}{\phi} Y \xto{q} X$, $(g,y) \mapsto tg$ is a \emph{topological quotient}: it is onto, and a subset of $X$ is open iff so is its inverse image in $Z$.
 \item The (isometric) homomorphism $G \times_X B \to Z \times_X A$, $(g,b) \mapsto \bigl(g,pb;g\chi(b)\bigr)$ of Banach algebra bundles over $Z$ is a (bicontinuous) isomorphism.
\end{enumerate}

\begin{cmts*} It follows from PF1 that $\phi$ must be fully faithful as an abstract functor and from PF2 that it must be essentially surjective. In consequence of Remark \ref{rmks:open}\textit{b} below, as applied to the evident action of $H \tto Y$ on $Z = G \ftimes{s}{\phi} Y$ along its projection to $Y$, the map $q$ in PF2 must actually be an \emph{open} surjection. (A continuous, surjective, and open map is always a topological quotient map but in general a topological quotient map need not be open.) The conjunction of the first two properties PF1–2 constitutes a weakening of the concept of “topological weak equivalence”, which would only be applicable as long as the existence of continuous local cross-sections to $q$ were granted through each point of $Z$. While the latter turns out to be too strong a requirement for our purposes, our weaker hypothesis PF2 happens to be satisfied in the following important cases, which will be of practical interest in the next section. \end{cmts*}

\begin{exms}\label{exms:P2} {\itshape (a)}\spacefactor3000\/ Let us call a topological groupoid $G \tto X$ \emph{transitive}, resp., \emph{locally transitive}, if its anchor map $(t,s): G \to X \times X$ is onto, resp., open. When a locally compact groupoid $G \tto X$ is both transitive and locally so, the inclusion homomorphism of any isotropy group $G(y) \subset G$ satisfies PF1 and PF2 above with $H = G(y)$, and it follows easily from our previous comments that the converse is also true. Let now $G$ be a locally compact group acting transitively from the left on a locally compact Hausdorff space $X$. Local transitivity of the transformation groupoid $G \times X \tto X$ is equivalent to openness of the map $G \to X$, $g \mapsto gy$, or, if we let $H \subset G$ denote the closed subgroup stabilizing $y$, to the continuous bijection $G/H \to X$, $gH \mapsto gy$ being a homeomorphism (by a result of Glimm, the latter map must be so whenever $G$ and $X$ are both second countable \cite[Thm.~1]{Gli61}).

{\itshape (b)}\spacefactor3000\/ If a locally compact groupoid $G \tto X$ is proper [cf.~the text preceding the statement of Theorem \ref{thm:almost}], as well as transitive, it is necessarily also locally transitive, and our considerations in Example \ref{exms:P2}\textit{a} apply. To prove it, all we need to do in virtue of the same considerations is to show that every subset $U \subset X$ for which $t^{-1}(U) \cap G_y$ is open in $G_y$ must be open in $X$. Let us fix $x$ in $U$, and let us consider the directed family of all its compact neighborhoods $C$. The sets $t^{-1}(C) \cap G_y \smallsetminus t^{-1}(U)$ are compact by properness and form a directed family. Their intersection is empty. Hence $t^{-1}(C) \cap G_y \subset t^{-1}(U)$, equivalently by transitivity $C \subset U$, for at least one $C$. Since $x$ was arbitrary, our claim is proven. In connection with our previous comments concerning PF2, we point out that in spite of our present assumptions there are no reasons why the open map $G_y \to X$, $g \mapsto tg$ should admit continuous local cross-sections. Going back to the case [considered in Example \ref{exms:P2}\textit{a}] of a locally compact group $G$ acting transitively on a locally compact Hausdorff space $X$, the problem of their existence translates into “given $H \subset G$ a closed subgroup, does the quotient projection $G \to G/H$ admit continuous local cross-sections?” The circumstance that it does is occasionally expressed by saying that \emph{$H$ has a local cross-section}. Now, a well-known theorem of Mostert \cite{Mos53} says that $H$ has a local cross-section whenever $G$ is separable, metrizable, and finite-dimensional, equivalently, $G$ is a projective limit of Lie groups all having the same dimension; another theorem of Gleason \cite{Gle50} says that $H$ has a local cross-section whenever it is isomorphic to a compact Lie group; on the other hand, according to Steenrod \cite[§7.5, p.~33]{Ste51}, “An unpublished example of Hanner provides a compact abelian group of infinite dimension and a closed $0$-dimensional subgroup without a local cross-section.”

{\itshape (c)}\spacefactor3000\/ Let $G \tto X$ be an arbitrary locally compact groupoid. Let $Y \subset X$ be an open subset such that $GY = X$. Take $H = G \mathbin| Y \tto Y$ (with the subspace topology), and let $\phi$ be the inclusion. Then $H \tto Y$ is locally compact and PF2 is satisfied (as well as obviously PF1) because, by local compactness, $t$ is an open map from $G$ and, hence, $s^{-1}(Y)$, onto $X$. As before, the map $q$ of PF2 is open but may fail to admit continuous local cross-sections. \end{exms}

\begin{cmts*}[\textit{continued}] Our last hypothesis, PF3, is equivalent to the fiberwise injectivity of $\chi$ plus the existence of a continuous map $Z \times_X A \to B$, $(g,y;a) \mapsto a^{g,y}$ covering the projection $Z = G \ftimes{s}{\phi} Y \xto\pr Y$ and satisfying \(%
	a = g \cdot \chi(a^{g,y})
\). Of course $(g,y;a) \mapsto (g,a^{g,y})$ is the inverse of the isomorphism referred to in PF3. For instance, in the situations of the above examples, taking $B = A \mathbin| Y$ and $\chi$ to be the inclusion map, PF3 is automatically satisfied with $a^{g,y} = g^{-1}a$. We shall make tacit use of a few obvious identities: \(%
	{a_1}^{g,y}{a_2}^{g,y} = (a_1a_2)^{g,y}\) (multiplicativity), \(%
	(g_1a)^{g_1g_2,y} = a^{g_2,y}\) (left invariance), \(%
	ha^{g\phi h,sh} = a^{g,th}
\) (right invariance). We shall in addition make use of the following: \end{cmts*}

\begin{rmks}\label{rmks:open} {\itshape (a)}\spacefactor3000\/ Every continuous action $Z \times H \to Z$, $(z,h) \mapsto zh$ of a topological group $H$ on a topological space $Z$ is an open map because the continuous involution of $Z \times H$ given by $(z,h) \mapsto (zh,h^{-1})$ transforms the action into the projection $Z \times H \to Z$, which is open. More generally, let an arbitrary topological groupoid with open target map $H \tto Y$ act from the right on the topological space $Z$ in a continuous fashion along a continuous map $Z \xto{p} Y$. In this case too the action $Z \ftimes{p}{t} H \to Z$, $(z,h) \mapsto zh$ must be an open map: the continuous involution of $Z \ftimes{p}{t} H$ defined as before transforms the action into the projection $Z \ftimes{p}{t} H \xto{\pr} Z$, so we have to make sure the latter is open; since we topologize $Z \ftimes{p}{t} H$ as a subspace of $Z \times H$, we only need to show that for all open $V \subset Z$, $W \subset H$ the image of $V \ftimes{p}{t} W$ under $\pr$ is open in $Z$; now given $z$ in that image, since by hypothesis $t(W) \ni pz$ is open in $Y$, we are able to find $U \ni z$ open inside $V$ such that $p(U) \subset t(W)$ and, therefore, $U = \pr(U \ftimes{p}{t} W) \subset \pr(V \ftimes{p}{t} W)$.

{\itshape (b)}\spacefactor3000\/ Under the same hypotheses as in Remark \ref{rmks:open}\textit{a}, the topological quotient map $q: Z \to Z/H$~(= the space of $H$-orbits $zH$) is open; for if $V \subset Z$ is open, then $q^{-1}\bigl(q(V)\bigr) = VH$ is the invariant saturation of $V$, that is, the image of $V \ftimes{p}{t} H$ under the action (which as we have seen is open), and so $q(V)$ is open in $Z/H$ by definition of the quotient topology.

{\itshape (c)}\spacefactor3000\/ Let $X' \xto{p} X$ and $Y' \xto{q} Y$ be continuous maps of spaces $X' \to B \leftarrow Y'$, $X \to B \leftarrow Y$ continuously fibered over a given base $B$. If each one of $p$,\ $q$ is open, or surjective, then obviously so must be $X' \times_B Y' \xto{p\times q} X \times_B Y$. \end{rmks}

\paragraph*{The direct image construction.} We are done with preliminaries now and ready to explain how to obtain a $\sigma$-representation of $G \tto X$ out of any given $\tau$-representation of $H \tto Y$, say, $S: s^*F \simto t^*F$ operating in a Banach $B$-module bundle $F \to Y$. Remember that we require $S$ to be \emph{isometric}: \(%
	\lvert S(h)f\rvert = \lvert f\rvert
\) for all $h \in H$, $f \in F_{sh}$.

We first pull the Banach bundle $F \xto{p} Y$ back along the continuous map $G \ftimes{s}{\phi} Y \xto{\pr} Y$; the pullback bundle projection \(%
	\pr^*F \simeq G \ftimes{s}{\phi p} F \xto{\id\times p} G \ftimes{s}{\phi} Y = Z
\) is an open map. We have a continuous right action of $H \tto Y$ on $Z$ along $Z \xto{\pr} Y$ given by \(%
	(g,th) \cdot h = (g\phi h,sh)
\), and we claim our next formula defines a right $H \tto Y$ action on $\pr^*F$ relative to which $\pr^*F \to Z$ becomes a $H \tto Y$ equivariant map.
\begin{equation}
	(g,f) \cdot h = \bigl(g\phi h,S(h)^{-1}\sigma(g,\phi h)^{g,th}f\bigr)
\label{eqn:direct1}
\end{equation}
Let us verify that our formula does indeed define a groupoid action. The identity $(g,f) \cdot 1pf = (g,f)$ follows from \eqref{eqn:unital} for $S$ and from the normality of $\sigma$. In virtue of the $B$-linearity of $S$, of \eqref{eqn:multiplicative} for $S$, of the cocycle equation \eqref{eqn:2-cocycle} for $\sigma$, and of our hypothesis about how $\sigma$ and $\tau$ relate,
\begin{alignat*}{2}
 \bigl((g,f) \cdot h_1\bigr) \cdot h_2 &
	= \bigl(g\phi h_1\phi h_2,S(h_2)^{-1}\sigma(g\phi h_1,\phi h_2)^{g\phi h_1,th_2}S(h_1)^{-1}\sigma(g,\phi h_1)^{g,th_1}f\bigr) \\ &
	= \bigl(g\phi h_1\phi h_2,S(h_2)^{-1}S(h_1)^{-1}h_1\sigma(g\phi h_1,\phi h_2)^{g\phi h_1,sh_1}\sigma(g,\phi h_1)^{g,th_1}f\bigr) \\ &
	= \bigl(g\phi(h_1h_2),[\tau(h_1,h_2)S(h_1h_2)]^{-1}\sigma(g\phi h_1,\phi h_2)^{g,th_1}\sigma(g,\phi h_1)^{g,th_1}f\bigr) \\ &
	= \bigl(g\phi(h_1h_2),S(h_1h_2)^{-1}\tau(h_1,h_2)^{-1}[g\sigma(\phi h_1,\phi h_2)\sigma(g,\phi h_1\phi h_2)]^{g,th_1}f\bigr) \\ &
	= \bigl(g\phi(h_1h_2),S(h_1h_2)^{-1}\tau(h_1,h_2)^{-1}\tau(h_1,h_2)\sigma(g,\phi h_1\phi h_2)^{g,th_1}f\bigr) \\ &
	= (g,f) \cdot h_1h_2.
\end{alignat*}
This concludes our verification. The bundle projection $\pr^*F \to Z$ descends to a map between the spaces obtained by quotienting by the $H \tto Y$ action, \[%
	E := (\pr^*F)/H \longto Z/H \simeq X, \quad [g,f] \longmapsto tg,
\] which has to be continuous (by the universal property of the quotient topology) and surjective (by PF2). We contend it has to be \emph{open} as well. To see it, since $\pr^*F \to Z$ is open, it will be enough to show $Z \to X$, $(g,y) \mapsto tg$ is open. The latter map is constant along the $H$-orbits and hence factors as $Z \to Z/H \to X$, where the quotient projection $Z \to Z/H$ is open by Remark \ref{rmks:open}\textit{b} above and where $Z/H \to X$ is bijective by PF1 plus PF2 and in fact a homeomorphism by the universal property of the quotient topology.

We proceed to turn this continuous open surjection $E \to X$ into a \emph{normed vector bundle} i.e.~a topological vector bundle endowed with a continuous norm function. Clearly \eqref{eqn:direct1} defines an $H \tto Y$ action on $\pr^*F \to Z$ by fiberwise linear maps and so there exists a unique algebraic vector bundle structure on $E \to X$ for which the quotient projection $\pr^*F \to E$ becomes a fiberwise linear map covering $Z \to X$. From Remark \ref{rmks:open}\textit{c} above plus our observation that every continuous open surjection is a topological quotient map it follows that this algebraic structure $\{+: E \times_X E \to E$,~$\cdot: \R \times E \to E\}$ is continuous and so $E \to X$ is a topological vector bundle. Next, because of our requirement that $S$ be isometric, the standard Banach norm \(%
	\lvert(g,f)\rvert = \lvert f\rvert
\) on $\pr^*F \to Z$ is invariant under the $H \tto Y$ action \eqref{eqn:direct1}, in other words, this action is by linear isometries. There must therefore be a unique norm function, $\lvert\blank\rvert: E \to \R_{\geq 0}$, making $\pr^*F \to E$ a fiberwise isometry. This function is necessarily continuous by the universal property of the quotient topology.

Because of the existence of a continuous norm function, the total space $E$ of our topological vector bundle $E \to X$ must be Hausdorff. Because of the existence of isometric identifications $E_{tg} \simeq F_y$, its fibers must be Banach spaces. Thus, in order to be able to say that $E \to X$ is a Banach bundle, we only need to verify it enjoys the “net property”. Let $e_i$ be any net in $E$ for which $\lvert e_i\rvert \to 0$, $x_i = pe_i \to x \in X$. Let $Z \xto{q} X$, $(g,y) \mapsto tg$ be the \emph{open} continuous surjection referred to in PF2. Given $z \in q^{-1}(x)$, the pairs $(W,i)$ with $W$ an open neighborhood of $z$ in $Z$ and $x_i$ in $q(W)$ form a \emph{directed} poset under the partial ordering \[%
	(W_1,i_1) \geq (W_2,i_2) \text{\quad iff\quad
		$W_1 \subset W_2$ and $i_1 \geq i_2$.}
\] For each $(W,i)$ let us pick $z_{W,i} = (g_{W,i},y_{W,i})$ in $q^{-1}(x_i) \cap W$ and put $e_i = [g_{W,i},f_{W,i}]$. Clearly $z_{W,i} \to z$ in $Z$ (by openness of $q$) and hence $(g_{W,i},f_{W,i}) \to 0z$ in $\pr^*F$ because $\lvert (g_{W,i},f_{W,i})\rvert = \lvert e_i\rvert \to 0$ and because $\pr^*F \to Z$ is a Banach bundle. Since continuous images of convergent nets are convergent, we conclude that $e_i \to 0x$.

A continuous bilinear operation $A \times_X E \to E$ satisfying $\lvert ae\rvert \leq \lvert a\rvert\,\lvert e\rvert$ and turning $E \to X$ into a Banach $A$-module bundle will be furnished by the formula
\begin{equation}
	a[g,f] = [g,a^{g,pf}f]
\label{eqn:direct2}
\end{equation}
once we make sure that the latter is meaningful in other words that setting $a \ast (g,f) = (g,a^{g,pf}f)$ endows $\pr^*F \to Z$ with an $H \tto Y$ equivariant Banach module bundle structure for $q^*A \to Z$ the Banach algebra bundle pullback of $A \to X$ along the map $Z \xto{q} X$ of PF2; now, by the $B$-linearity of $S$ and by the centrality of $\sigma$, we do in fact have
\begin{align*}
 a \ast \bigl((g,f) \cdot h\bigr) &
	= \bigl(g\phi h,a^{g\phi h,sh}S(h)^{-1}\sigma(g,\phi h)^{g,th}f\bigr) \\ &
	= \bigl(g\phi h,S(h)^{-1}ha^{g\phi h,sh}\sigma(g,\phi h)^{g,th}f\bigr) \\ &
	= \bigl(g\phi h,S(h)^{-1}[a\sigma(g,\phi h)]^{g,th}f\bigr) \\ &
	= \bigl(g\phi h,S(h)^{-1}\sigma(g,\phi h)^{g,th}a^{g,th}f\bigr) \\ &
	= \bigl(a \ast (g,f)\bigr) \cdot h.
\end{align*}

We define a $\sigma$-representation $R: s^*E \simto t^*E$ of $G \tto X$ on our Banach $A$-module bundle $E \to X$ by setting
\begin{equation}
	R(g')[g,f] = [g'g,\sigma(g',g)^{g'g,pf}f].
\label{eqn:direct3}
\end{equation}
Clearly this formula will yield a morphism in $\BB(G)$ provided $g' \cdot (g,f) = \bigl(g'g,\sigma(g',g)^{g'g,pf}f\bigr)$ gives an $H \tto Y$ equivariant left $G \tto X$ operation on $\pr^*F \to Z$, which it does in consequence of the $B$-linearity of $S$ and of the cocycle equation \eqref{eqn:2-cocycle} for $\sigma$:
\begin{align*}
 g' \cdot \bigl((g,f) \cdot h\bigr) &
	= \bigl(g'g\phi h,\sigma(g',g\phi h)^{g'g\phi h,sh}S(h)^{-1}\sigma(g,\phi h)^{g,th}f\bigr) \\ &
	= \bigl(g'g\phi h,S(h)^{-1}h\sigma(g',g\phi h)^{g'g\phi h,sh}\sigma(g,\phi h)^{g,th}f\bigr) \\ &
	= \bigl(g'g\phi h,S(h)^{-1}[\sigma(g',g\phi h)g'\sigma(g,\phi h)]^{g'g,th}f\bigr) \\ &
	= \bigl(g'g\phi h,S(h)^{-1}\sigma(g'g,\phi h)^{g'g,th}\sigma(g',g)^{g'g,pf}f\bigr) \\ &
	= \bigl(g' \cdot (g,f)\bigr) \cdot h.
\end{align*}
The morphism $R: s^*E \to t^*E \in \BB(G)$ thus obtained is $A$-linear as a result of the centrality of $\sigma$:
\begin{align*}
 g' \cdot \bigl(a \ast (g,f)\bigr) &
	= \bigl(g'g,\sigma(g',g)^{g'g,pf}a^{g,pf}f\bigr) \\ &
	= \bigl(g'g,[\sigma(g',g)g'a]^{g'g,pf}f\bigr) \\ &
	= \bigl(g'g,[g'a]^{g'g,pf}\sigma(g',g)^{g'g,pf}f\bigr) \\ &
	= g'a \ast \bigl(g' \cdot (g,f)\bigr).
\end{align*}
It satisfies \eqref{eqn:unital} because of the normality of $\sigma$, as well as satisfying \eqref{eqn:multiplicative} substantially by reason of the cocycle equation \eqref{eqn:2-cocycle} for $\sigma$:
\begin{align*}
 g'' \cdot \bigl(g' \cdot (g,f)\bigr) &
	= \bigl(g''g'g,\sigma(g'',g'g)^{g''g'g,pf}\sigma(g',g)^{g'g,pf}f\bigr) \\ &
	= \bigl(g''g'g,[\sigma(g'',g'g)g''\sigma(g',g)]^{g''g'g,pf}f\bigr) \\ &
	= \bigl(g''g'g,\sigma(g'',g')^{g''g'g,pf}\sigma(g''g',g)^{g''g'g,pf}f\bigr) \\ &
	= \sigma(g'',g') \ast \bigl(g''g' \cdot (g,f)\bigr).
\end{align*}
Our assumptions imply in addition that $R: s^*E \simto t^*E$ is an isometric $\sigma$-representation. As a matter of notation, we shall be writing $E = \phi_*F$, $R = \phi_*S$.

Let $M: F \to F'$ be an intertwiner of isometric $\tau$-representations of $H \tto Y$. In view of \eqref{eqn:intertwiner} and of the $B$-linearity of $M$, setting \(%
	L_{tg}[g,f] = [g,M_{pf}f]
\) gives a well-defined fiberwise linear map $L = \phi_*M: \phi_*F \to \phi_*F'$ and, in fact, an intertwiner of isometric $\sigma$-representations (its continuity being an immediate consequence of the universal property of the quotient topology). The correspondence $S \mapsto \phi_*S$ thus defined from isometric $\tau$-representations of $H \tto Y$ to isometric $\sigma$-representations of $G \tto X$ is functorial. We leave it to the reader to explicitly write down the prospected natural isometric equivalences $\phi_*\phi^*R \simeq R$, $\phi^*\phi_*S \simeq S$.

On top of our overall assumptions let us now suppose that $A \to X$, $B \to Y$ are equivariant Banach $\ast$-algebra bundles, $\chi$ is a $\ast$-homomorphism, and the two multipliers $\sigma$,\ $\tau$ are unitary. The reader will have no difficulties implementing the obvious changes to the above when $F \to Y$ is a \emph{Hilbert} $B$-module bundle on which $S: s^*F \simto t^*F$ operates \emph{unitarily}. That being the case, it will be possible to endow $E \to X$ with a canonical Hilbert $A$-module bundle structure with respect to which $R$ will operate unitarily, and $\phi_*\phi^*R \simeq R$, $\phi^*\phi_*S \simeq S$ will be natural unitary equivalences.

\begin{cmts*} Inasmuch as nonprojective unitary representations of transitive locally compact groupoids are concerned, a precursor of the above direct image construction appears in Seda's thesis \cite[p.~158, Thm.~5.4.12]{Sed74}. However, Seda's definition of “Hilbert bundle” is somewhat weaker than ours: not only may a Hilbert bundle in his sense fail to have continuous local cross-sections, it may not even be open as a map. This explains why we have to work harder in order to show that our construction produces a Banach or Hilbert bundle. Also, in [loc.~cit.], the uniqueness and local triviality of the direct image “Hilbert bundle” are established on condition that the groupoid itself be locally trivial; for the reasons discussed in our Examples \ref{exms:P2}\textit{a}–\textit{c}, we cannot afford any hypothesis to that effect. A different approach to the Morita theory of groupoid representations, alternative to both Seda's and ours, is described in \cite[arXiv version, §3.6]{Bos11}. It appears that in order to make sense of his variant of the direct image construction Bos has to postulate the existence of “Dirac approximate identities”---some kind of “distributional cross-sections” to the map mentioned in our condition PF2 above. The absence of such technical complications in our treatment looks like a clear advantage of the Banach bundle viewpoint adopted in this paper. \end{cmts*}

\section{The local extension theorem}\label{sec:local}

In this section we specialize down to the case of multipliers for $G \tto X$ with values in a Banach algebra $\Ban{A}$, i.e., in the trivial Banach algebra $G$-bundle $X \times \Ban{A} \xto{\pr} X$, $g \cdot (sg,a) = (tg,a)$; in practice our interest lies mainly in the complex case, $\Ban{A} = \C$, but we do not want to exclude other possible applications. We let $\End_{\Ban{A}}(\Ban{E}) \subset \L(\Ban{E},\Ban{E})$ denote the Banach algebra of all bounded $\Ban{A}$-linear operators $T$ in a Banach $\Ban{A}$-module $\Ban{E}$, and $\GL_{\Ban{A}}(\Ban{E})$ the group of all invertible such $T$. We say that $T \in \GL_{\Ban{A}}(\Ban{E})$ is an \emph{isometry}, and write $T \in \GI_{\Ban{A}}(\Ban{E})$~(= the \emph{Banach} or \emph{isometry group} of $\Ban{E}$), if it is norm preserving: $\lvert Te\rvert = \lvert e\rvert$ for all $e \in \Ban{E}$. When $\Ban{A}$ is a Banach $\ast$-algebra and $\Ban{E}$ is a Hilbert $\Ban{A}$-module [cf.~the beginning of section \ref{sec:images}], we have the subgroup $\GU_{\Ban{A}}(\Ban{E}) \subset \GI_{\Ban{A}}(\Ban{E})$ of all \emph{unitary}, i.e., inner-product preserving, $\Ban{A}$-linear automorphisms: $\langle Te,Tf\rangle = \langle e,f\rangle$ for all $e$,~$f \in \Ban{E}$. We call $\GU_{\Ban{A}}(\Ban{E})$ the \emph{Hilbert} or \emph{unitary group} of $\Ban{E}$. Our next theorem is the main result of this paper.

\begin{thm}\label{thm:local} Let\/ $G \tto X$ be a proper groupoid\/ {\upshape [cf.~section \ref{sec:almost}]}, let\/ $\Ban{A}$ be a Banach algebra, and let\/ $\sigma$ be an\/ $\Ban{A}$-valued multiplier for\/ $G \tto X$ satisfying\/ $\lvert\sigma(g,h)^{-1}\rvert = 1/\lvert\sigma(g,h)\rvert$. For any uniformly continuous\/ $\sigma$-representation\/ $S: G(x) \to \GL_{\Ban{A}}(\Ban{E})$ of the isotropy group of\/ $G \tto X$ at\/ $x \in X$ on a Banach\/ $\Ban{A}$-module\/ $\Ban{E}$, in a sufficiently small open neighborhood\/ $V$ of\/ $x$ there exists some uniformly continuous\/ $\sigma$-representation\/ $G \mathbin| V \to \GL_{\Ban{A}}(\Ban{E})$ of\/ $G \mathbin| V \tto V$ on\/ $\Ban{E}$ which extends\/ $S$. \end{thm}

The proof will be given presently but first let us make a few comments and record some more or less immediate corollaries of our theorem.

\begin{rmks*} In the differentiable case [see Example (d) at the beginning of section \ref{sec:representations}] the linearization theorem for proper Lie groupoids provides a sufficiently good understanding of the local structure of $G \tto X$ to enable a relatively straightforward resolution of the local extension problem (at least for nonprojective representations): over a small enough slice $Z$ through $x$, the restricted groupoid $G \mathbin| Z \tto Z$ must be isomorphic to the transformation groupoid associated with the canonical linear action of the isotropy group $G(x)$ on the tangent space $T_xZ$, and for any such groupoid the extension problem has an evident solution which can be further prolonged over an open neighborhood of $x$ by means of a standard Morita argument \cite[pp.~833–835 plus Ex.~2.12]{2008b}. On the other hand, when $G \tto X$ is purely topological, the concepts involved in the previous reasoning do not even make sense. We do not know whether in our theorem the hypothesis that the isotropy representation be uniformly continuous can be done away with: we use this hypothesis crucially in the course of our proof on at least two distinct occasions. Likewise, we ignore whether for $S: G(x) \to \GI_{\Ban{A}}(\Ban{E})$ operating by isometries there is always some local extension which also operates by isometries; nevertheless, in the Hilbert case, we are able to say the following: \end{rmks*}

\begin{cor}\label{thm:local*} Let\/ $\Ban{A} = \R$,~$\C$,~or\/ $\mathbb{H}$, and let\/ $\sigma$ be a unitary\/ $\Ban{A}$-valued multiplier for the proper groupoid\/ $G \tto X$. Let\/ $S: G(x) \to \GU_{\Ban{A}}(\Ban{E})$ be a uniformly continuous unitary\/ $\sigma$-representation of the isotropy group of\/ $G \tto X$ at\/ $x \in X$ on a Hilbert\/ $\Ban{A}$-module\/ $\Ban{E}$. Over an open neighborhood\/ $V$ of\/ $x$, there is a uniformly continuous unitary\/ $\sigma$-representation\/ $G \mathbin| V \to \GU_{\Ban{A}}(\Ban{E})$ extending\/ $S$. \end{cor}

\begin{proof}[Proof of the corollary] By Theorem \ref{thm:local} there is some uniformly continuous $S: G \mathbin| V \to \GL_{\Ban{A}}(\Ban{E})$ which at $x$ agrees with the given isotropy representation. Possibly after taking $V$ smaller around $x$ so that $G \mathbin| V \tto V$ itself admits cutoff functions [see section \ref{sec:almost}], we may pretend with no loss of generality that $V$ is all of $X$ and thus $S$ is defined on all of $G$. We can change the $\Ban{A}$-valued inner product $\langle\blank,\blank\rangle$ on $\Ban{E}$ to a parametric family $\{\left\langle\blank,\blank\right\rangle_v$,~$v \in X\}$ of such under which $S$ operates by unitaries on $X \times \Ban{E} \to X$, namely, in the notations introduced after the proof of Lemma \ref{lem:inverse},
\begin{equation*}
	\langle e,f\rangle_v = \integral_{th=v} c(sh)\langle S(h^{-1})e,S(h^{-1})f\rangle \der h.
\end{equation*}
Because of the uniform continuity of $S$, as a family $X \to \L^2(\Ban{E},\Ban{A})$~(= the Banach space of all real-bilinear maps $\Ban{E} \times \Ban{E} \to \Ban{A}$), the new $\Ban{A}$-valued inner products $\langle\blank,\blank\rangle_v$ depend continuously on $v$. By the same token, for every $v$, the norm $\lvert\blank\rvert_v$ associated with $\langle\blank,\blank\rangle_v$ is equivalent to the norm $\lvert\blank\rvert$ associated with $\langle\blank,\blank\rangle$: for some $b > 0$, we have \(%
	\lvert e\rvert^2/b^2 \leq \langle S(h^{-1})e,S(h^{-1})e\rangle \leq b^2\lvert e\rvert^2
\) for all those $h \in G^v$ for which $c(sh) > 0$, whence $\lvert e\rvert/b \leq \lvert e\rvert_v \leq b\lvert e\rvert$.

We proceed to use the “unitarian trick”, that is, letting $I_v = U_vP_v$ for each $v$ in $X$ denote the polar decomposition of the identity map of $\Ban{E} = (\Ban{E},\langle\blank,\blank\rangle)$ to $E_v = (\Ban{E},\langle\blank,\blank\rangle_v)$ into a hilbertian $\Ban{A}$-linear isomorphism $U_v$ of $\Ban{E}$ to $E_v$ and a hermitian positive definite $\Ban{A}$-linear operator $P_v$ on $\Ban{E}$ (this being legitimate since $I_v$ is a topological $\Ban{A}$-linear isomorphism between the two Hilbert $\Ban{A}$-modules $\Ban{E}$ and $E_v$), we set, for every $g$ in $G$, \[%
	\tilde{S}(g) = U_{tg}^{-1} \circ S(g) \circ U_{sh} \in \GU_{\Ban{A}}(\Ban{E}).
\] In view of the continuous dependence of $\langle\blank,\blank\rangle_v$ on $v$ the operators $U_v$,\ $P_v$ must themselves depend continuously on $v$. The only issue is whether $\tilde{S}$ is still an extension of $S$ at $x$, in other words, whether $\tilde{S}(k) = S(k)$ for every $k$ in $G(x)$. Now, since $S(k)$ was unitary relative to the original inner product $\langle\blank,\blank\rangle$, by the uniqueness of polar decomposition we have \(%
	U_x = S(k)^{-1} \circ U_x \circ S(k)\) or, equivalently, \(%
	U_x^{-1} \circ S(k) \circ U_x = S(k)
\), which settles our issue affirmatively. \end{proof}

Before turning to the proof of our local extension theorem we want to discuss how to globalize our last two results. The demonstration of both the local and the global version of Theorem \ref{thm:local} will rely on an auxiliary lemma which we are about to state now but whose proof we shall only give later in the course of the section. By a \emph{Banach $\Ban{A}$-module bundle} we shall mean a Banach $A$-module bundle for the trivial Banach algebra $G$-bundle $A = X \times \Ban{A} \to X$. Given any Banach $\Ban{A}$-module $\Ban{E}$, by a \emph{$\GL_{\Ban{A}}(\Ban{E})$-structure} (resp., \emph{$\GI_{\Ban{A}}(\Ban{E})$-structure}) for a Banach $\Ban{A}$-module bundle $E \to X$ we shall mean a $\GL(\Ban{E})$-structure [cf.~the last part of section \ref{sec:representations}] that contains some atlas consisting exclusively of $\Ban{A}$-linear (resp., $\Ban{A}$-linear and isometric) local trivializations $U \times \Ban{E} \simto E \mathbin| U \in \BB_{\Ban{A}}(U)$. We shall accordingly speak of \emph{$\GL_{\Ban{A}}(\Ban{E})$-structured}, resp., \emph{$\GI_{\Ban{A}}(\Ban{E})$-structured} Banach $\Ban{A}$-module bundles.

\begin{lem}\label{lem:direct} Suppose that a continuous homomorphism\/ $\phi: H \to G$ of locally compact groupoids\/ $H \tto Y$, $G \tto X$ satisfies both assumptions\/ {\upshape PF1} and\/ {\upshape PF2} of section\/ {\upshape \ref{sec:images}}. Let\/ $\sigma$ be an isometric\/ $\Ban{A}$-valued multiplier for\/ $G \tto X$, and let\/ $\phi^*\sigma$ be the\/ $\Ban{A}$-valued multiplier for\/ $H \tto Y$ given by\/ \(%
	\phi^*\sigma(h,k) = \sigma(\phi h,\phi k)
\); the latter automatically satisfies\/ {\upshape PF3} plus our other requirements in section\/ {\upshape \ref{sec:images}}. Let\/ $S: s^*F \simto t^*F \in \BB(H)$ be a uniformly continuous isometric\/ $\phi^*\sigma$-representation of\/ $H \tto Y$ on a\/ $\GL_{\Ban{A}}(\Ban{E})$-structured Banach\/ $\Ban{A}$-module bundle\/ $F \to Y$. Provided either
\begin{enumerate}
\def\labelenumi{\upshape (\roman{enumi})}
 \item the map\/ $G \ftimes{s}{\phi} Y \to X$ of\/ {\upshape PF2} admits continuous local cross-sections or
 \item the locally compact groupoids\/ $H \tto Y$, $G \tto X$ are proper,
\end{enumerate}
there is, for the direct image Banach\/ $\Ban{A}$-module bundle\/ $E = \phi_*F \to X$, a canonical\/ \emph{direct image $\GL_{\Ban{A}}(\Ban{E})$-structure} relative to which both\/ $\phi^*\phi_*F \simeq F \in \BB(Y)$ and\/ $\phi_*S: s^*E \simto t^*E \in \BB(G)$ are uniformly continuous. \end{lem}

\begin{cmts*} Whenever a $\sigma$-representation $R: s^*E \simto t^*E \in \BB(G)$ is uniformly continuous relative to a given $\GL_{\Ban{A}}(\Ban{E})$-structure for $E \to X$, its pullback $\phi^*R$ is uniformly continuous relative to the “inverse image” $\GL_{\Ban{A}}(\Ban{E})$-structure for $\phi^*E \to Y$. This is straightforward but the converse direction is much less so. In the differentiable setting the theory of Morita equivalence for representations is not problematic insofar as local triviality is concerned \cite[§4.3]{2008a} because there is a large supply of continuous local cross-sections. On the other hand, as we saw in Example \ref{exms:P2}\textit{b}, in the general locally compact setting we can no longer take (i) for granted; it is here that the properness hypothesis (ii) comes to rescue as a surrogate of (i). As far as we can tell, the direct image of an $\GI_{\Ban{A}}(\Ban{E})$-structure might no longer be itself an $\GI(\Ban{E})$-\ let alone $\GI_{\Ban{A}}(\Ban{E})$-structure, not even locally around $\phi(Y)$. \end{cmts*}

Our proof of either version of Theorem \ref{thm:local} will also rely on a second auxiliary lemma.

\begin{lem}\label{lem:multiplier} Let\/ $\sigma$ be an\/ $\Ban{A}$-valued multiplier for a proper groupoid\/ $G \tto X$. Suppose it satisfies\/ \(%
	\lvert\sigma(g,h)^{-1}\rvert = 1/\lvert\sigma(g,h)\rvert
\). (Automatic when\/ $\Ban{A} = \R$,~$\C$,~or\/ $\mathbb{H}$.) Then, there is a continuous positive real-valued function\/ $\rho$ on\/ $G$ such that the following is an isometric\/ $\Ban{A}$-valued multiplier for\/ $G \tto X$. \[%
	\tilde{\sigma}(g,h) := \frac{\rho(g)\rho(h)}{\rho(gh)}\sigma(g,h)
\] \end{lem}

\begin{proof} The function $\lvert\sigma\rvert > 0$ is a multiplier with positive real values. Setting $\beta(g,h) = \log{}\lvert\sigma(g,h)\rvert$ defines a real-valued \emph{additive} $2$-cocycle on $G$. As this takes values in a vector space, by a standard cohomology vanishing argument it must be cohomologically trivial: there is some real-valued $1$-cochain $\alpha$ on $G$ such that \(%
	\beta(g,h) = \alpha(h) - \alpha(gh) + \alpha(g)
\) for instance \(%
	\alpha(g) = \integral c(sh)\beta(g,h) \der h
\) in the notations of section \ref{sec:almost}. Setting \(%
	\rho(g) = \exp\bigl(-\alpha(g)\bigr)
\) then furnishes the desired function. \end{proof}

By a \emph{local presentation} of a Banach $\Ban{A}$-module bundle $E \to X$ by a Banach $\Ban{A}$-module $\Ban{E}$ over an open subset $U$ of $X$ we mean an $\Ban{A}$-linear morphism \(%
	U \times \Ban{E} \to E \mathbin| U \in \BB_{\Ban{A}}(U)
\) that is fiberwise split onto. We say that $E \to X$ is \emph{locally presentable} (\emph{of type\/ $\Ban{E}$}) if the domains $U$ of such local presentations cover $X$, and that it is (locally presentable) \emph{of finite type} if $\Ban{E}$ is finitely generated as an $\Ban{A}$-module.

\begin{cor}\label{thm:global} Under the assumptions of\/ {\upshape Theorem \ref{thm:local}}, there is a\/ $\sigma$-representation of\/ $G \tto X$ on a locally presentable Banach\/ $\Ban{A}$-module bundle\/ $E \to X$ of type\/ $\Ban{E}$ whose restriction to the isotropy group of\/ $G \tto X$ at\/ $x$ is a\/ $\sigma$-representation\/ $G(x) \to \GL_{\Ban{A}}(E_x)$ equivalent to\/ $S$. \end{cor}

\begin{proof} Lemma \ref{lem:multiplier} tells us that $\tilde{\sigma} = \sigma/\lvert\sigma\rvert$ is an isometric $\Ban{A}$-valued multiplier for $G \tto X$ cohomologous to $\sigma$ via a $1$-cochain $\rho: G \to \R_{>0}$. Let $\tilde{S}: G(x) \to \GL_{\Ban{A}}(\Ban{E})$ be the uniformly continuous $\tilde{\sigma}$-representation obtained by setting $\tilde{S}(k) = \rho(k)S(k)$. Given (to within equivalence) any extension of $\tilde{S}$ to a $\tilde{\sigma}$-representation of the sort indicated in the statement, an extension of $S$ to an analogous $\sigma$-representation will be obtained simply by dividing back by $\rho$. We may thus pretend without loss of generality that our multiplier $\sigma$ is isometric.

Let $S: G \mathbin| V \to \GL_{\Ban{A}}(\Ban{E})$ be an arbitrary local extension of $S$ to a uniformly continuous $\sigma$-representation of $G \mathbin| V \tto V$ on a trivial Banach $\Ban{A}$-module bundle $V \times \Ban{E} \to V$, the existence of which is guaranteed by Theorem \ref{thm:local}. Possibly after taking a smaller $V$ so as to ensure that $G \mathbin| V \tto V$ itself admits cutoff functions [see section \ref{sec:almost}], we can by averaging change the constant norm function on the latter Banach bundle to a nonconstant, locally equivalent one that is still continuous but under which $S$ now acts by isometries [cf.~the proof of Corollary \ref{thm:local*}]. The identity map $V \times \Ban{E} \xto{=} F$ serves as a global $\Ban{A}$-linear trivialization for the Banach $\Ban{A}$-module bundle $F \to V$ resulting from this new choice of norm function. We are now in a situation where we can apply Lemma \ref{lem:direct} to the uniformly continuous isometric $\sigma$-representation $S$ of $H = G \mathbin| V \tto V$ on $F \to V$ [cf.~Example \ref{exms:P2}\textit{c}]. To within equivalence, our lemma then yields an extension of $S$ to a uniformly continuous $\sigma$-representation $R$ of $G \mathbin| U \tto U$ on a $\GL_{\Ban{A}}(\Ban{E})$-structured locally trivial Banach $\Ban{A}$-module bundle $E \to U$ of class $C^0$, where $U = GV$ is the invariant saturation of $V$.

Let us fix some continuous $G$-invariant function $c: X \to [0,1]$ with $c(x) = 1$, $\supp c \subset U$. (Such functions can be obtained by a standard averaging argument which we leave to the reader.) Let us redefine $E_u$ to be $0$ wherever $c(u) = 0$, and, among all cross-sections of this new family of Banach spaces $E_u$, let us consider those of the form $u \mapsto c(u)\zeta(u)$, for all $\zeta \in \Gamma(U;E)$. Clearly, these form a vector space, have continuous norm functions, and span every $E_u$. By a well-known criterion \cite[p.~102, Prop.~10.4]{Fel77}, they are continuous for a unique topology on the family $\{E_u\}$ that turns this family into a Banach bundle \(%
	E = \bigcup E_u \to X
\). This is a locally presentable Banach $\Ban{A}$-module bundle (obviously) to which $R$ descends continuously (say, by Lemma \ref{lem:morphism}). \end{proof}

Notice that the representation $R$ constructed in the preceding proof operates by isometries on the locally presentable Banach $\Ban{A}$-module bundle $E \to X$ whenever $\lvert\sigma\rvert = 1$. On the other hand, upon assuming $S: G(x) \to \GI_{\Ban{A}}(\Ban{E})$ isometric, we are at present unable to conclude that $R$ is also necessarily \emph{isometrically} equivalent at $x$ to $S$, except under the following circumstances:

\begin{cor}\label{thm:global*} Under the hypotheses of\/ {\upshape Corollary \ref{thm:local*}}, there exists some unitary\/ $\sigma$-representation of\/ $G \tto X$ on a locally presentable Hilbert\/ $\Ban{A}$-module bundle\/ $E \to X$ of type\/ $\Ban{E}$ which, at\/ $x$, induces an isotropy\/ $\sigma$-representation unitarily equivalent to\/ $S$. \end{cor}

\begin{proof} This follows from Corollary \ref{thm:local*} and Lemma \ref{lem:direct} by an argument practically identical to the one laid out in the preceding proof. \end{proof}

\begin{cor}\label{thm:enough} On any proper groupoid, the representative functions arising from finite-type complex unitary representations are sufficiently many to separate arrows. \end{cor}

\begin{proof} Immediate from our last corollary and the classical result for compact groups (\cite[p.~21, Cor.~4.4]{Bre72}, \cite[p.~134, Thm.~3.1]{BtD85}). \end{proof}

\subsection*{Proof of the local extension theorem}

\begin{proof}[Proof of\/ {\upshape Lemma \ref{lem:direct}}] We shall only provide an argument under the simplifying hypothesis that there is some global $\Ban{A}$-linear trivialization $Y \times \Ban{E} \simto F \in \BB_{\Ban{A}}(Y)$ relative to which $S$ can be expressed as an operator-norm continuous map $S: H \to \GL_{\Ban{A}}(\Ban{E})$ satisfying \(%
	S(h)S(k) = \sigma(\phi h,\phi k)S(hk)
\), as for us in practice this is the only relevant case. The general case will be left to the reader.

Let $Z := G \ftimes{s}{\phi} Y \xto{q} X$, $(g,y) \mapsto tg$ be the map of PF2, and let $Z \xto{\pr} Y$, $(g,y) \mapsto y$ be the projection. The pullback $\pr^*F \to Z$ is itself globally trivial as a Banach $\Ban{A}$-module bundle, and $H \tto Y$ acts on $\pr^*F \to Z$ from the right by the rule \(%
	(g,th;e) \cdot h = \bigl(g\phi h,sh;S(h)^{-1}\sigma(g,\phi h)e\bigr)
\), where $e \in \Ban{E}$. For an arbitrary open $U \subset X$ one may change the Banach $\Ban{A}$-module bundle trivialization of $\pr^*F \to Z$ locally over $W = q^{-1}(U)$ to a $C^0$-compatible one, namely, to \(%
	W \times \Ban{E} \simto (\pr^*F) \mathbin| W \in \BB_{\Ban{A}}(W)\) given by \(%
		(g,y;e) \mapsto \bigl(g,y;T(g,y)e\bigr)
\), by using a continuous map $W \xto{T} \GL_{\Ban{A}}(\Ban{E})$. The new local trivialization will not descend to one of $E := \pr^*F/H \to X$ over $U$, in other words, the problem \[%
\xymatrix{%
	W \times \Ban{E} \ar[r]^-\sim
	\ar[d]^{q\times\id}
	&	(\pr^*F) \mathbin| W
		\ar[d]^{-/H}
\\	U \times \Ban{E} \ar@{-->}[r]^-\sim_-?
	&	E \mathbin| U
}\] will not admit a solution (of necessity unique and bicontinuous), unless $(g,th;e)$ and $(g\phi h,sh;e)$ always have the same image in $E$ under the top-to-bottom composition, equivalently,
\begin{equation}
	S(h)T(g\phi h,sh) = \sigma(g,\phi h)T(g,th)
\label{eqn:direct}
\end{equation}
for all $h$, equivalently, $T$ is $H$-equivariant; clearly just taking $T(g,y) = \id$ for all $g$,\ $y$ won't work. With the proviso that $X$ can be covered with open sets $U$ for which we can find such $T$'s over each $q^{-1}(U)$, we have the following\spacefactor3000: (1)~Any two of the resulting local trivializations $U \times \Ban{E} \simto E \mathbin| U$ of $E \to X$ are $C^0$-compatible, hence these trivializations give a well-defined $\GL_{\Ban{A}}(\Ban{E})$-structure for $E \to X$. (2)~Both the isometry $\phi^*E \simeq F$ and the direct image $\sigma$-representation $\phi_*S: s^*E \simto t^*E$ are uniformly continuous relative to the $\GL_{\Ban{A}}(\Ban{E})$-structure in question.

Now, whenever the map $q$ of PF2 admits continuous local cross-sections, the existence of enough suitable $T$'s can be established as follows. Let $U \to Z$, $u \mapsto \bigl(g(u),y(u)\bigr)$ be any such cross-section: $tg(u) = u$, $sg(u) = \phi y(u)$. By PF1, $(g,y) \in W = q^{-1}(U)$ equals $\bigl(g(tg),y(tg)\bigr) \cdot h$ for exactly one $h = h(g,y)$ satisfying $th = y(tg)$ and depending continuously on $(g,y)$. We obtain $T$ on $W$ by setting \(%
	\bigl(g,y;T(g,y)e\bigr) = \bigl(g(tg),y(tg);e\bigr) \cdot h(g,y)
\), equivalently,
\begin{equation*}
	T(g,y) = S\bigl(h(g,y)\bigr)^{-1}\sigma\bigl(g,\phi h(g,y)\bigr).
\end{equation*}

On the other hand, whenever both $H \tto Y$ and $G \tto X$ are proper, an adequate stock of suitable $T$'s can be obtained by means of the following procedure. Given $x_0 \in X$, let us fix $z_0 = (g_0,y_0) \in Z$ so as to have $tg_0 = x_0$. There is a unique $H$-equivariant map $T_0: q^{-1}(x_0) \to \GL_{\Ban{A}}(\Ban{E})$ for which $T_0(z_0) = \id$; this is defined along the $H$-orbit $q^{-1}(x_0) = z_0H$, and is given by \(%
	T_0(g_0\phi h,sh) = S(h)^{-1}\sigma(g_0,\phi h)
\) for all $h \in H^{y_0}$. By the uniform continuity of $S$, this map is also continuous from the closed subset $q^{-1}(x_0)$ of $Z$ to the Banach space $\End_{\Ban{A}}(\Ban{E})$. By Fell's version of Tietze's extension theorem \cite[p.~104, Thm.~10.7]{Fel77} there has to be some continuous extension $Z \to \End_{\Ban{A}}(\Ban{E})$ of $T_0$ to all of $Z$. By using a normalized Haar measure system on $H \tto Y$, we can average this out and make it $H$-equivariant as well without changing it along $q^{-1}(x_0)$. Now, by the properness of $G \tto X$ and consequently of the map $Z = G \times_X Y \xto{q} X$, there must be some open neighborhood $U$ of $x_0$ in $X$ for which on the open $H$-invariant “tube” $W = q^{-1}(U)$ the continuous $H$-equivariant extension $Z \to \End_{\Ban{A}}(\Ban{E})$ of $T_0$ actually takes values in $\GL_{\Ban{A}}(\Ban{E})$. \end{proof}

\begin{rmk*} Within the framework of the preceding proof, whenever $\Ban{E}$ is a Hilbert $\Ban{A}$-module for $\Ban{A} = \R$,~$\C$,~or $\mathbb{H}$ and $S: H \to \GU_{\Ban{A}}(\Ban{E})$ is a unitary $\phi^*\sigma$-representation, we may consider the polar decomposition $T(g,y) = T'(g,y)P(g,y)$ of $W \xto{T} \GL_{\Ban{A}}(\Ban{E})$. It is easy to see that $W \xto{T'} \GU_{\Ban{A}}(\Ban{E})$ must be continuous. $T'$ must satisfy \eqref{eqn:direct} as well because, upon substituting into \eqref{eqn:direct}, uniqueness of polar decomposition implies equality of the unitary parts of the two members of the resulting equation. Under the present circumstances we can thus actually find enough $\GU_{\Ban{A}}(\Ban{E})$-valued $T$'s and, consequently, an $\GI_{\Ban{A}}(\Ban{E})$-structure for $E \to X$. \end{rmk*}

\begin{proof}[Proof of\/ {\upshape Theorem \ref{thm:local}}] For the same reasons we put forward in the proof of Corollary \ref{thm:global}, we may suppose with no loss of generality that $\sigma$ is an isometric multiplier and that $S: G(x) \to \GI_{\Ban{A}}(\Ban{E})$ is an isometric $\sigma$-representation; all we need to do is replace $\sigma$ with $\tilde{\sigma} = \sigma/\lvert\sigma\rvert$ and $S$ with the uniformly continuous $\tilde{\sigma}$-representation $\tilde{S}(g) = \rho(g)S(g)$, where $\rho$ is as in Lemma \ref{lem:multiplier}, and then pass to an equivalent norm under which $\tilde{S}$ acts isometrically on the Banach $\Ban{A}$-module $\Ban{E}$; it is perhaps of interest to note that in the present case the existence of such norms is a consequence of the boundedness of $\tilde{S}$ alone \cite[corollary p.~82]{Lyu88} rather than of the compactness of $G(x)$.

Let $O = Gx$ be the $G$-orbit through $x$. By properness this is a closed invariant subspace of $X$. The restriction $G \mathbin| O \tto O$ is itself a locally compact groupoid, in fact, a proper one: since it is also transitive, we are in the situation of Example \ref{exms:P2}\textit{b}. We may therefore apply Lemma \ref{lem:direct} to the inclusion $H = G(x) \subset G \mathbin| O$ and to the $\sigma$-representation $S: G(x) \to \GI_{\Ban{A}}(\Ban{E})$ so as to find an extension of $S$ to a uniformly continuous $\sigma$-representation $R: s^*E \simto t^*E$ of $G \mathbin| O \tto O$ on a $\GL_{\Ban{A}}(\Ban{E})$-structured Banach $\Ban{A}$-module bundle $E \to O$: over some open neighborhood $U$ of $x$ in $O$, there exists some Banach $\Ban{A}$-module bundle trivialization $U \times \Ban{E} \simto E \mathbin| U \in \BB_{\Ban{A}}(U)$ relative to which $R$ is represented by a continuous $\sigma$-multiplicative map $R: G \mathbin| U \to \GL_{\Ban{A}}(\Ban{E})$ extending $S: G(x) \to \GI_{\Ban{A}}(\Ban{E})$. Let us write $U = O \cap V$, where $V$ is an open neighborhood of $x$ in $X$. Viewing $G \mathbin| U = G \mathbin| O \cap G \mathbin| V$ as a closed subspace of $G \mathbin| V$ and $R$ as a continuous partial cross-section of the trivial Banach bundle with fiber $\End_{\Ban{A}}(\Ban{E})$ over $G \mathbin| V$, we may invoke Fell's version of Tietze's extension theorem \cite[p.~104, Thm.~10.7]{Fel77} in order to extend $R: G \mathbin| U \to \GL_{\Ban{A}}(\Ban{E})$ to a continuous map $T: G \mathbin| V \to \End_{\Ban{A}}(\Ban{E})$.

We claim that, possibly after shrinking $V$ around $x$ if necessary, any such extension $T$ must be an almost $\sigma$-representation of $G \mathbin| V \tto V$ on the trivial Banach $\Ban{A}$-module bundle $V \times \Ban{E} \to V$. Our argument is a straightforward adaptation of \cite[§4.1]{2016a}. To begin with, since our claim is local around $x$, there is no loss of generality in pretending that $T: G \to \End_{\Ban{A}}(\Ban{E})$ is a continuous map defined on all of $G \tto X$, that its restriction $G \mathbin| O \to \GL_{\Ban{A}}(\Ban{E})$ along $O = Gx$ is $\sigma$-multiplicative, and that the continuous function $g \mapsto \lVert T(g)\rVert$ is bounded by, say, $b \geq 1$ on $G$. Let us consider the following two open subsets of $G$. (In order to be able to say they are open, we need the continuity of $T: G \to \End(\Ban{E})$; this is the second time we fall back on our uniform continuity hypothesis about the isotropy representation $S: G(x) \to \GL(\Ban{E})$; the first time was when we invoked \cite[p.~104, Thm.~10.7]{Fel77}.)
\begin{align*}
	\varOmega_0&= \{y \in X: \lVert\id - T(1y)\rVert < b^{-2}/18\} \\
	\varOmega_2&= \{(g,h) \in G \ftimes{s}{t} G: \lVert\sigma(g,h)T(gh) - T(g)T(h)\rVert < b^{-2}/18\}
\end{align*}
Clearly $x \in \varOmega_0$, $G(x) \times G(x) \subset \varOmega_2$. If we are able to find an open subset $V \subset \varOmega_0$ containing $x$ such that $G \mathbin| V \ftimes{s}{t} G \mathbin| V \subset \varOmega_2$, then $T: G \mathbin| V \to \GL_{\Ban{A}}(\Ban{E})$ will be an almost $\sigma$-representation: since $\sigma$ is isometric and therefore $\ell(\sigma,g) = 1$, for all $v$ in any such $V$ we shall have \[%
	r(T,v) \leq
	\sup_{y\in V}{}\lVert\id - T(1y)\rVert + \sup_{\substack{g,h\in G\mathbin|V\\ sg=th}}\lVert\sigma(g,h)T(gh) - T(g)T(h)\rVert
	\leq b^{-2}/9 \leq b(T,v)^{-2}/9.
\] Now by properness of $G \tto X$ for every compact neighborhood $C$ of $x$ the set $G \mathbin| C$ is compact in $G$ in $\varOmega_0$ and thus $G \mathbin| C \ftimes{s}{t} G \mathbin| C \smallsetminus \varOmega_2$ is compact in $G \ftimes{s}{t} G$. The latter compact sets have empty intersection because obviously the intersection of all $G \mathbin| C \ftimes{s}{t} G \mathbin| C$ is equal to $G(x) \times G(x)$. Then $G \mathbin| C \ftimes{s}{t} G \mathbin| C \smallsetminus \varOmega_2$ will be empty for some $C$, whose interior will serve as our $V$.

We are now in a position to conclude. By our almost representation theorem, Theorem \ref{thm:almost}, as applied to $G \mathbin| V \tto V$, there exists some uniformly continuous $\sigma$-representation $G \mathbin| V \to \GL_{\Ban{A}}(\Ban{E})$ lying for each $v$ in $V$ within a distance of $4b(T,v)r(T,v)$ from $T$ along $Gv \cap V$. Since $T$ restricted along $U = Gx \cap V$ is the $\sigma$-representation $R$ and so $r(T,x) = r(R,x) = 0$, this must coincide with $R$ along $U$ and thus, a fortiori, with the original representation $S: G(x) \to \GI_{\Ban{A}}(\Ban{E})$ at $x$. The proof of our local extension theorem, Theorem \ref{thm:local}, is finished. \end{proof}

\section{Tannakian reconstruction of proper groupoids}\label{sec:tannakian}

Let $G \tto X$ be a locally compact groupoid. Let $\Rep(G)$ denote the category of all its finite-type complex unitary representations and their intertwiners. The notion of complex tensor category, which we proceed to review, captures much of the relevant structure of $\Rep(G)$.

A \emph{symmetric monoidal structure} on a category $\cat{C}$ comprises a bifunctor $\otimes: \cat{C} \times \cat{C} \to \cat{C}$ (the \emph{monoidal product}), a distinguished object $1 \in \cat{C}$ (the \emph{monoidal unit}), and natural isomorphisms
\begin{alignat*}{2} &
	R \otimes (S \otimes T) \simeq (R \otimes S) \otimes T &\quad &\text{(the \emph{associativity constraint})} \\ &
	R \otimes S \simeq S \otimes R                         &\quad &\text{(the \emph{symmetry constraint})} \\ &
	1 \otimes R \simeq R \simeq R \otimes 1                &\quad &\text{(the \emph{unit constraints})}
\end{alignat*}
satisfying certain axioms a detailed knowledge of which will be dispensable for our purposes; we refer the reader to \cite[Ch.~XI.1]{ML98} for a discussion of these axioms, which are known as (\emph{Mac~Lane's}) \emph{coherence conditions}.

Let $\cat{C} = (\cat{C},\otimes,1,\mathellipsis)$ be a \emph{symmetric monoidal category}, i.e., a category endowed with a symmetric monoidal structure. A \emph{symmetric monoidal functor} of $\cat{C}$ to another symmetric monoidal category $\cat{D} = (\cat{D},\otimes,1,\mathellipsis)$ is the result of adjoining natural isomorphisms
\begin{align*}
	\varPhi(R) \otimes \varPhi(S)&\simeq \varPhi(R \otimes S) & 1&\simeq \varPhi(1) & &\text{(the \emph{monoidal functor constraints})}
\end{align*}
to a functor $\varPhi: \cat{C} \to \cat{D}$ in such a manner that a few obvious requirements of compatibility with the symmetric monoidal constraints are satisfied which we are not going to reproduce here; the reader is referred to \cite[Ch.~XI.2]{ML98} for a precise formulation of these requirements.

Suppose that each hom-set $\cat{C}(R,S)$ of a given symmetric monoidal category $\cat{C}$ is equipped with the structure of a complex vector space in such a way that the monoidal product bifunctor $\otimes$ and the composition of morphisms are both bilinear. We obtain a \emph{conjugation} in $\cat{C}$ by specifying a symmetric monoidal endofunctor $R \mapsto \cnj{R}$ of $\cat{C}$ which is antilinear on hom-sets plus a natural isomorphism \(%
	\cnj{\cnj{R}} \simeq R
\) for which the composition \(%
	\cnj{R} \simeq \cnj{\cnj{(\cnj{R})}} = \cnj{(\cnj{\cnj{R}})} \simeq \cnj{R}
\) is the identity. We call a symmetric monoidal category $\cat{C}$ equipped with a compatible complex-linear structure and some such conjugation a \emph{complex tensor category} if, as a preadditive category, $\cat{C}$ is additive i.e.~admits a zero object $0$ and direct sum objects $R \oplus S$ for all $R$,~$S \in \cat{C}$; as the complex tensor categories of interest to us will mostly happen not to be abelian, this is as much as we demand.

Let $\cat{C}$, $\cat{D}$ be complex tensor categories. We upgrade the concept of a symmetric monoidal functor $\varPhi: \cat{C} \to \cat{D}$ to that of a \emph{complex tensor functor} by adding one more constraint to the list, namely, a natural isomorphism \(%
	\cnj{\varPhi(R)} \simeq \varPhi(\cnj{R})
\) such that the composition \(%
	\varPhi(R) \simeq \cnj{\cnj{\varPhi(R)}} \simeq \cnj{\varPhi(\cnj{R})} \simeq \varPhi(\cnj{\cnj{R}}) \simeq \varPhi(R)
\) equals the identity, and by requiring that, as an ordinary functor, $\varPhi$ be complex-linear on hom-sets.

\begin{exms*} Let $\Hil(X)$ denote the category of all finite-type complex Hilbert bundles over a locally compact space $X$; by definition, $\Hil(X)$ is a full subcategory of $\HB_\C(X)$. For any locally compact groupoid $G \tto X$, let $\Rep(G)$ denote the category of all complex unitary representations of $G \tto X$ on objects of $\Hil(X)$. Both $\Hil(X)$ and $\Rep(G)$ can be promoted to complex tensor categories in an evident fashion. The pullback $\Rep(G) \to \Rep(H)$ along any continuous homomorphism of locally compact groupoids $H \to G$, as well as the forgetful functor $\Rep(G) \to \Hil(X)$, is a complex tensor functor. \end{exms*}

Let $\varPhi$,~$\varPsi: \cat{C} \to \cat{D}$ be two complex tensor functors of a complex tensor category $\cat{C}$ to another such category $\cat{D}$. Let $\gamma: \varPhi \hto \varPsi$ be a natural transformation. We express the circumstance that \(%
	1 \simeq \varPhi(1) \xto{\gamma_1} \varPsi(1)
\) is equal to $1 \simeq \varPsi(1)$ and that%
\begin{multline*}
	\varPhi(R) \otimes \varPhi(S) \xto{\gamma_R\otimes\gamma_S} \varPsi(R) \otimes \varPsi(S) \simeq \varPsi(R \otimes S) \text{\quad equals} \\%
		\varPhi(R) \otimes \varPhi(S) \simeq \varPhi(R \otimes S) \xto{\gamma_{R\otimes S}} \varPsi(R \otimes S) \text{\quad $\forall R$,~$S \in \cat{C}$}
\end{multline*}
by saying $\gamma$ is \emph{tensor preserving}, in symbols, $\gamma \in \Hom^\otimes(\varPhi,\varPsi)$. We further call $\gamma \in \Hom^\otimes(\varPhi,\varPsi)$ \emph{self-conjugate}, and write $\gamma \in \Hom^{\cnj\otimes}(\varPhi,\varPsi)$, if%
\begin{equation*}
	\cnj{\varPhi(R)} \xto{\cnj{\gamma_R}} \cnj{\varPsi(R)} \simeq \varPsi(\cnj{R}) \text{\quad equals}\quad%
		\cnj{\varPhi(R)} \simeq \varPhi(\cnj{R}) \xto{\gamma_{\cnj{R}}} \varPsi(\cnj{R}) \text{\quad $\forall R \in \cat{C}$.}
\end{equation*}

\begin{lem}\label{lem:tannakian} Let\/ $\cat{C}$ be a complex tensor category and let\/ $\varPhi: \cat{C} \to \VSp$ be a complex tensor functor with values in the complex tensor category of all finite-dimensional complex vector spaces. Let\/ $\pi: K \to \End^{\cnj\otimes}(\varPhi)$ be a homomorphism of a compact group\/ $K$ to the monoid of all self-conjugate tensor-preserving natural transformations of\/ $\varPhi$ into itself. Suppose that\/ $\pi$ is\/ \emph{continuous} in the sense that\/ $\pi_R: K \to \GL\bigl(\varPhi(R)\bigr)$, $k \mapsto \pi(k)_R$ is a continuous map---and hence\/ $\varPhi(R)$ is a finite-dimensional complex\/ $K$-module---for every object\/ $R$ of $\cat{C}$. Suppose in addition that the two conditions hereafter are satisfied:
\begin{enumerate}
\def\labelenumi{\upshape(\roman{enumi})}
 \item Each finite-dimensional complex\/ $K$-module can be embedded\/ $K$-equivariantly into\/ $\varPhi(R)$ for at least one object\/ $R \in \cat{C}$.
 \item Each\/ $K$-equivariant linear map\/ $\varPhi(R) \to \varPhi(S)$ is equal to\/ $\varPhi(L)$ for at least one morphism\/ $R \xto{L} S \in \cat{C}$.
\end{enumerate}
Then\/ $\pi: K \simto \End^{\cnj\otimes}(\varPhi)$ is bijective, and hence\/ $\End^{\cnj\otimes}(\varPhi) = \Aut^{\cnj\otimes}(\varPhi)$ is not only a monoid, but actually a group. \end{lem}

\begin{proof} Let $K\text-\VSp$ denote the category of all finite-dimensional complex $K$-modules and their $K$-equivariant linear transformations. Let $\varPsi: K\text-\VSp \to \VSp$ be the forgetful functor, which we regard as a complex tensor functor in the evident way. We are going to construct an isomorphism of complex algebras \[%
	\End(\varPhi) \simto \End(\varPsi),~\gamma \mapsto \tilde{\gamma}
\] under which $\End^{\cnj\otimes}(\varPhi)$ is going to correspond bijectively to $\End^{\cnj\otimes}(\varPsi)$ and whose composition with $\pi$ is going to be equal to the standard “tautological” isomorphism $\varepsilon: K \simto \End^{\cnj\otimes}(\varPsi)$ of classical Tannaka duality---the one sending $k \in K$ to $\varepsilon(k)$ given by left multiplication by $k$ \cite{BtD85,JS91}. Clearly our lemma will follow from this.

Let $\gamma \in \End(\varPhi)$ be an arbitrary natural transformation of $\varPhi$ to itself. By (i), for each finite-dimensional complex $K$-module $V$ there exists an object $R \in \cat{C}$ together with a $K$-equivariant linear embedding $V \xto{\iota} \varPhi(R)$. Since finite-dimensional complex $K$-modules for a compact $K$ are reductive, there exists some $K$-equivariant linear map $\varPhi(R) \xto{\varrho} V$ such that $\varrho \circ \iota = \id$ on $V$. We contend that the linear map \(%
	\varrho \circ \gamma_R \circ \iota: V \to V
\) does not depend on our choices of $R$, $\iota$, and $\varrho$. To see it, let us suppose slightly more generally that we are given a $K$-equivariant linear map $\beta: V \to W$ to another finite-dimensional complex $K$-module $W$ plus a pair of $K$-equivariant linear maps $W \xto{\kappa} \varPhi(S) \xto{\varsigma} W$ for some $S \in \cat{C}$ which form a retraction in the sense that $\varsigma \circ \kappa = \id$ on $W$. Now the composition \(%
	\varPhi(R) \xto{\varrho} V \xto{\beta} W \xto{\kappa} \varPhi(S)
\) is a $K$-equivariant linear map $\varPhi(R) \to \varPhi(S)$ and so by (ii) it must be equal to $\varPhi(L)$ for some morphism $R \xto{L} S$ in $\cat{C}$. By the naturality of $\gamma$ we must then have
\begin{equation*}
 \varsigma \circ \gamma_S \circ \kappa \circ \beta
	= \varsigma \circ \gamma_S \circ \varPhi(L) \circ \iota
	= \varsigma \circ \varPhi(L) \circ \gamma_R \circ \iota
	= \beta \circ \varrho \circ \gamma_R \circ \iota.
\end{equation*}
Upon taking $\beta$ to be the identity map $V \xto= V$, our contention follows. Our reasoning shows more generally that the correspondence $V \mapsto \tilde{\gamma}_V := \varrho \circ \gamma_R \circ \iota$ is natural in $V$ and thus defines an element $\tilde{\gamma}$ of $\End(\varPsi)$.

For any $V \xto{\iota} \varPhi(R)$ as above, by definition of the $K$-module structure on $\varPhi(R)$ and by the $K$-equivariance of $\iota$, we have \(%
	\widetilde{\pi(k)}_V = \varrho \circ \pi(k)_R \circ \iota = \varrho \circ \iota \circ \varepsilon(k)_V = \varepsilon(k)_V
\). We leave it to the reader to check that the correspondence $\End(\varPhi) \simto \End(\varPsi)$, $\gamma \mapsto \tilde{\gamma}$ is bijective, as well as linear and multiplicative, and that $\gamma$ is self-conjugate (resp., tensor preserving) iff so is $\tilde{\gamma}$. \end{proof}

\begin{cmts*} The preceding lemma generalizes \cite[Prop.~2.3]{2008a} to the case where $K$ is compact though not necessarily a Lie group. The assumptions of loc.~cit.~do not include (i) because in the Lie group case the latter hypothesis may be dropped at the cost of losing the injectivity of $\pi$; one can invoke the “descending chain property” for closed normal subgroups of $K$ (see \cite[p.~136]{BtD85}, \cite[p.~429]{JS91}) to show that there has to be some object $R_0$ of $\cat{C}$ for which the kernel of the representation $\pi_{R_0}$ is equal to $\bigcap_{R\in\cat{C}} \ker\pi_R$ and that for any such $R_0$ every finite-dimensional complex $K$-module whose kernel contains $\bigcap_{R\in\cat{C}} \ker\pi_R$ has to be a submodule of $\varPhi(R_0)$. \end{cmts*}

\subsection*{Statement and proof of the reconstruction theorem}

Let us now consider the forgetful functor $\Rep(G) \xto{\varPhi} \Hil(X)$, $R = (E,R) \mapsto E$ and, for each $x$ in $X$, let us write $\varPhi_x: \Rep(G) \to \VSp$ for its composition with the functor $\Hil(X) \to \VSp$ that assigns $E$ the corresponding fiber at $x$ regarded as a plain finite-dimensional complex vector space---in other words $\varPhi_x(R)$ is what is left of $E_x$ after discarding the hermitian inner product. There is an evident way of promoting $\varPhi_x$ to a complex tensor functor.

\begin{defn}\label{defn:tannakian} The \emph{tannakian bidual} of a locally compact groupoid $G \tto X$ is the groupoid \(%
	\Bid(G) \tto X
\) with hom-sets \[%
	\Bid(G)(x,y) = \Iso^{\cnj\otimes}(\varPhi_x,\varPhi_y) \text{,\quad $x$,~$y \in X$}
\] (the right-hand side being the set of all self-conjugate tensor preserving natural isomorphisms of $\varPhi_x$ to $\varPhi_y$) and with multiplication law \(%
	(\delta\gamma)_R = \delta_R\gamma_R
\) given by the standard “vertical” composition of natural transformations \cite[p.~42]{ML98}. Its arrow space $\Bid(G)$ is given the coarsest topology for which all “representative functions” \[%
	\gamma \mapsto \langle\gamma_R\zeta(s\gamma),\eta(t\gamma)\rangle
		\text{,\quad $R = (E,R) \in \Rep(G)$, $\zeta$,~$\eta \in \Gamma(X;E)$}
\] are continuous. It is easy to check that the topology in question turns $\Bid(G) \tto X$ into a topological groupoid and that the following homomorphism, the \emph{comparison map}, is continuous: \[%
	\pi: G \longto \Bid(G),~g \longmapsto \{R \mapsto R(g)\}.
\] \end{defn}

Although for a general $G \tto X$ the comparison map need not be injective or surjective, the following is true.

\begin{thm}\label{thm:tannakian} For any proper\/ $G \tto X$ the comparison map is an isomorphism\/ $\pi: G \simto \Bid(G)$ of topological groupoids. \end{thm}

As a means toward demonstrating our theorem, in addition to Lemma \ref{lem:tannakian} and Corollary \ref{thm:enough}, we shall need the following.

\begin{lem}\label{lem:tannakian*} Let\/ $R = (E,R)$,~$S = (F,S)$ be objects of\/ $\Rep(G)$ for some proper groupoid\/ $G \tto X$. Given any\/ $G(x)$-equivariant linear map\/ $\lambda: E_x \to F_x$ between the fibers of\/ $E$,\ $F$ at\/ $x \in X$, one can find intertwiners\/ $L: R \to S \in \Rep(G)$ for which\/ $L_x = \lambda$. \end{lem}

\begin{proof}[Proof of the lemma] By properness, the orbit $O = Gx$ is a closed subspace of $X$. The restriction of $E$ along $O$ is a locally trivial Hilbert bundle $E \mathbin| O \to O$ of finite rank because, by the transitivity of $G \mathbin| O \tto O$, the dimension of its fibers is locally constant and so around each $y \in O$ any local frame of continuous cross-sections spanning $E_y$ constitutes a local trivialization. The same considerations apply to $F$. There is one and only one extension of $\lambda$ to an intertwiner $L: R \mathbin| O \to S \mathbin| O$, namely, \[%
	L_{tg} = S(g) \circ \lambda \circ R(g)^{-1} \text{,\quad $g \in G_x$;}
\] here, the right-hand side is independent of $g$ and thus defines a morphism $E \mathbin| O \to F \mathbin| O$ in $\Hil(O)$. Since $O \subset X$ is closed and $E \mathbin| O$,\ $F \mathbin| O$ are locally trivial, for some open $U \ni x$ this morphism can be extended to one of $E \mathbin| U$ to $F \mathbin| U$ in $\Hil(U)$; such extensions can be obtained by following orthogonal projection onto a local frame spanning $E$ at $x$ with an application of the ordinary Tietze extension theorem (for functions) to the continuous coefficients of a local matrix incarnation of our morphism $E \mathbin| O \to F \mathbin| O$.

So let $M: E \mathbin| U \to F \mathbin| U \in \Hil(U)$ be a local extension of $\lambda$ around $x$ which is $G$-equivariant along $O \cap U$. Let $c: X \to \R_{\geq 0}$ be a continuous function satisfying $c(x) > 0$, $\supp c \subset U$; possibly after rescaling it, we may suppose it satisfies \(%
	\integral_{th=x} c(sh) \der h = 1
\) as well. Let $\tilde{M}: E \to F \in \Hil(X)$ be given by $\tilde{M}_u = c(u)M_u$ for $u \in U$ and by zero elsewhere. Our considerations in section \ref{sec:almost} about Haar integration functionals \eqref{eqn:Haar} being valid not only when $c$ is a normalizing function but also more generally when it is a continuous function for which $t: \supp(c \circ s) \to X$ is a proper map, we define a new morphism $L: E \to F$ in $\Hil(X)$ by setting \[%
	L_ve = \integral_{th=v} c(sh)S(h)\tilde{M}_{sh}R(h)^{-1}e \der h \in F_v
		\text{,\quad $v \in X$, $e \in E_v$.}
\] It is straightforward to check that our $L$ intertwines $R$ and $S$ and that $L_x = \lambda$. \end{proof}

\begin{rmk*} There is no need to assume $X$ second countable or paracompact in order for our lemma to hold; our avoiding that kind of assumptions is the reason why the above proof happens to be slightly more technical than that of \cite[Lem.~5.1]{2008a}. \end{rmk*}

\begin{proof}[Proof of the theorem] In virtue of Corollary \ref{thm:enough} we know that the complex tensor functor $\varPhi_x: \Rep(G) \to \VSp$ and the continuous homomorphism $G(x) \to \End^{\cnj\otimes}(\varPhi_x)$ which $\pi$ induces upon restriction satisfy the condition (i) of Lemma \ref{lem:tannakian}; now thanks to our last lemma we also know they satisfy the other condition, (ii). Application of Lemma \ref{lem:tannakian} then guarantees that $\pi$ restricts to a bijection of $G(x)$ onto $\Bid(G)(x)$ for every $x$ in $X$. A simple argument similar to the one we employed in the proof of Lemma \ref{lem:tannakian*} shows that $G$-orbits can be separated by means of continuous invariant functions on $X$. As a consequence, since these functions are precisely the self-intertwiners of the trivial representation of $G \tto X$ on $X \times \C \to X$, we have $\Bid(G)(x,y) = \emptyset$ for all $x$,~$y \in X$ for which $G(x,y) = \emptyset$. The bijectivity of $\pi: G \to \Bid(G)$ is proven.

There remains to be shown that our continuous homomorphism $\pi$ is bicontinuous, equivalently, that its inverse is continuous, equivalently, that $\pi$ is an open map. So, let $\varOmega$ be an open neighborhood of $g$ in $G$: our claim is that $\pi(\varOmega)$ is a neighborhood of $\pi(g)$ in $\Bid(G)$. Let us write $x = sg$, $y = tg$. The set $g^{-1} \cdot [G(x,y) \cap \varOmega]$ is an open neighborhood of $1x$ in $G(x)$. By the classical representation theory of compact groups \cite[p.~21, Cor.~4.4]{Bre72}, for some finite-dimensional complex Hilbert space $\Ban{E}$ there exists some unitary representation $S: G(x) \to \GU(\Ban{E})$ with kernel contained in this open neighborhood equivalently with \(%
	g \cdot \ker S \subset \varOmega
\). By Corollary \ref{thm:global*} there exists some extension of $S$ to a complex unitary representation $R: s^*E \simto t^*E$ on a finite-type complex Hilbert bundle $E \to X$. Let $\zeta_1$,\ \textellipsis,\ $\zeta_n$, resp., $\eta_1$,\ \textellipsis,\ $\eta_n$, be continuous cross-sections of $E \to X$ that provide a vector basis $\zeta_1(x)$,\ \textellipsis,\ $\zeta_n(x)$, resp., $\eta_1(y)$,\ \textellipsis,\ $\eta_n(y)$, for $E_x$, resp., $E_y$. Put $a_{ij} = \langle R(g)\zeta_i(x),\eta_j(y)\rangle$. As long as it is possible to find any number $\epsilon > 0$ and compact neighborhoods $A \ni x$, $B \ni y$ for which the set \[%
	F^{A,B}_\epsilon = \{h \in G: \max_{i,j}{}\lvert\langle R(h)\zeta_i(sh),\eta_j(th)\rangle - a_{ij}\rvert \leq \epsilon,~sh \in A,~th \in B\}
\] is contained in $\varOmega$, by definition of the topology of $\Bid(G) \tto X$ and by the continuity of its source and target, $\pi(F^{A,B}_\epsilon)$ will be a neighborhood of $\pi(g)$ in $\Bid(G)$ and, because of the bijectivity of $\pi$, it will be contained in $\pi(\varOmega)$ so that, as desired, $\pi(\varOmega)$ will be a neighborhood of $\pi(g)$. In order to establish the existence of $\epsilon$,\ $A$,\ and $B$, we use the properness of $G \tto X$, as follows. Each set \(%
	F^{A,B}_\epsilon \smallsetminus \varOmega
\) is closed inside the compact set $s^{-1}(A) \cap t^{-1}(B)$ and so is itself compact. These compact sets obviously form a directed family under reverse inclusion. Their intersection is empty, for it is contained in \(%
	\{h \in G(x,y): R(h) = R(g)\} \smallsetminus \varOmega = g \cdot \ker S \smallsetminus \varOmega
\). Then $F^{A,B}_\epsilon \smallsetminus \varOmega = \emptyset$, equivalently, $F^{A,B}_\epsilon \subset \varOmega$, for at least one choice of $\epsilon$,\ $A$,\ and $B$. \end{proof}

\section{Appendix}\label{sec:appendix}

Our purpose is to make a number of interrelated considerations which provide additional insight into the theory discussed in the preceding sections and which further motivate it.

\begin{lem}\label{lem:appendix} Let\/ $G \xto{p} X$ be a topological group bundle with locally compact Hausdorff\/ $G$, and let\/ $\varphi \geq 0$ be a continuous function on\/ $G$ such that\/ $\supp\varphi \cap p^{-1}(C) \subset G$ is compact for every compact\/ $C \subset X$ and such that\/ $\varphi(1x) = 1$ for every\/ $x$ in\/ $X$. For each\/ $x$, let\/ $\mu^x$ be left Haar measure on\/ $G_x$ normalized so that\/ $\integral \varphi \der\mu^x = 1$. Then\/ $\mu = \{\mu^x\}$ is a (continuous) Haar measure system on\/ $G$. \end{lem}

\begin{proof} The reader is referred to \cite[p.~6, Lem.~1.3]{Ren91}. (The demonstration provided in loc.~cit.~is an adaptation of an argument attributed to Fell \cite[appendix]{Gli62}.) \end{proof}

Hence $G \xto{p} X$ is a \emph{locally compact group bundle}---in the sense of Example (c) of section \ref{sec:representations}. It is not hard to manufacture specimens of locally compact group bundles that are proper but are not locally trivial. For instance, let $G \subset \R^2$ be the closed subspace $\{(x,\pm x): x \in \R\}$, let $G \xto{p} X$ be the restriction of the first coordinate projection $\R^2 \to \R$, and let $G \times_X G \to G$ be the multiplication law given by \(%
	(x,ax) \cdot (x,bx) = (x,abx)
\), where $a$,~$b \in \{\pm 1\}$; then $p$ is open (in fact, it admits continuous cross-sections through each point of $G$), as well as continuous, and $\mu^x$ is counting measure for $x = 0$ and is half of it for $x \neq 0$. There are obvious variants of this example, say, with $G_x$ isomorphic to a circle group for $x \neq 0$. On the other hand, isotropy bundles of smooth actions of compact Lie groups on manifolds provide a host of examples of isotropic Hausdorff groupoids with locally compact arrow space and with proper, though not open, source map; for these, the system of normalized left Haar measures cannot possibly be a continuous one.

\begin{prop}\label{prop:appendix} Suppose a span of continuous homomorphisms of locally compact groupoids
\begin{equation}
 \begin{split}
\xymatrix{%
	G \ar@{<-}[r]^\phi
	\ar@<-.33ex>[d] \ar@<+.33ex>[d]
	&	H \ar[r]^-\psi
		\ar@<-.33ex>[d] \ar@<+.33ex>[d]
		&	K \times Z
			\ar@<-.33ex>[d] \ar@<+.33ex>[d]
\\	X \ar@{<-}[r]^\phi
	&	Y \ar[r]^\psi
		&	Z
}\end{split}
\label{eqn:appendix}
\end{equation}
is given for which the following hypotheses are satisfied:
\begin{enumerate}
\def\labelenumi{\upshape(\roman{enumi})}
 \item $\phi$ fulfills the conditions\/ \textup{PF1} and\/ \textup{PF2} of section\/ {\upshape \ref{sec:images}}.
 \item $G \tto X$ has proper anchor map\/ $(t,s): G \to X \times X$, its isotropy groups are (isomorphic to) Lie groups, and they form a topological group bundle over\/ $X$ that admits continuous local cross-sections through each arrow.
 \item $K \times Z \tto Z$ is a transformation groupoid for some locally compact group\/ $K$ which is either compact, or connected, or else a Lie group.
\end{enumerate}
Let\/ $Y_1$ be the set of all\/ $y \in Y$ for which\/ $\psi\bigl(H(y)\bigr) \subset \{1\} \times Z$, where\/ $H(y)$ is the isotropy group of\/ $H \tto Y$ at\/ $y$. Let\/ $X_1$ be the set of all\/ $x \in X$ such that\/ $x \in G \cdot \phi y$ for some\/ $y$ in\/ $Y_1$. Then\/ $X_1$ is a closed subset containing the entire connected component of each one of its points. \end{prop}

\begin{proof} Let us write $\psi(h) = \bigl(k(h),\psi(sh)\bigr)$. For any compact normal subgroup $N$ of $K$ such that the quotient $K/N$ is a Lie group, let $Y_N$ be the set of all $y \in Y$ such that $k\bigl(H(y)\bigr) \subset N$, and let $X_N$ be the image of $G \ftimes{s}{\phi} Y_N$ under the surjection \(%
	q: G \ftimes{s}{\phi} Y \to X$, $(g,y) \mapsto tg
\) of PF2; we claim that $X_N$ is a \emph{clopen}, i.e., simultaneously open and closed, subset of $X$. Once this is proven, our proposition will be proven too, on account of the following reasoning. Suppose $x_1 \in X_1$, $x \notin X_1$ are such that there exists some choice of $N$ for which $X_N$ does not contain $x$. Then, since $X_N$ is clopen and contains all of $X_1$, $x$ cannot possibly be in the connected component of $x_1$. Now, in order to find such an $N$ for any given $x_1 \in X_1$ and $x \notin X_1$, let us fix $y \in Y$ so that $x \in G \cdot \phi y$. By our hypothesis PF1 plus (ii), $H(y) \simeq G(\phi y)$ is a compact Lie group, and so is its homomorphic image $k\bigl(H(y)\bigr)$. Since a Lie group has no small subgroups, there must be $U \ni 1$ open in $K$ for which $k\bigl(H(y)\bigr) \cap U$ contains no subgroup of $k\bigl(H(y)\bigr)$ other than $\{1\}$. By the classical work of Gleason, Yamabe, and others on Hilbert's Fifth Problem \cite{Bor57}, we can find an open subgroup $K' \subset K$ and a compact normal subgroup $N$ of $K'$ contained in $U$ such that $K'/N$ is a Lie group. Thus, for $K$ connected, there is a compact normal $N$ contained in $U$ for which $K/N$ is a Lie group. The same conclusion is valid for compact $K$, thanks to \cite[p.~21, Cor.~4.4]{Bre72}, as well as for $K$ a Lie group, for trivial reasons. We necessarily have $k\bigl(H(y)\bigr) \cap N = \{1\}$ and hence $k\bigl(H(y)\bigr) \not\subset N$, on account of the nontriviality of $k\bigl(H(y)\bigr)$ implicit in our assumption that $x \notin X_1$ and, consequently, $y \notin Y_1$. Thus $y \notin Y_N$. But this evidently implies that $x \notin X_N$, since $Y_N$ is an invariant subset of $Y$ by the normality of $N$.

Let us now once and for all fix an arbitrary compact normal subgroup $N$ of $K$ for which $K/N$ is a Lie group, and write $\pi: K \to K/N$ for the quotient projection. By a well-known result of Montgomery and Zippin (cf.~\cite[p.~87, Cor.~5.6]{Bre72}, \cite[p.~216]{MZ55}), for any other compact subgroup $G$ of $K$ there exists some open neighborhood $\varOmega$ of $\pi(G)$ in $K/N$ with the property that each compact subgroup of $K/N$ contained in $\varOmega$ is conjugated to a subgroup of $\pi(G)$. Therefore, if $L \subset K$ is any compact subgroup contained in $W = \pi^{-1}(\varOmega) \supset G$, then $hLh^{-1} \subset GN$ for some $h$ in $K$. Because of the normality of $N$, upon taking $G = N$ we see in particular that for some open neighborhood $W$ of $N$ in $K$ it must be true that $L \subset N$ whenever $L \subset W$.

In order to prove that $X_N$ is open, it will suffice to prove that so is $Y_N$, for the former is the image of $G \ftimes{s}{\phi} Y_N$ under the open map $q$ of PF2. In virtue of the considerations we made in the previous paragraph, it will be enough to show that for any given $y_0 \in Y_N$ there is some open neighborhood $V$ of $y_0$ in $Y$ such that $k\bigl(H(v)\bigr) \subset W$ for all $v$ in $V$. Let $\Delta Y \subset Y \times Y$ be the diagonal, consisting of all pairs $(y,y)$. Its inverse image $a^{-1}(\Delta Y) \subset H$ under the anchor map $a: H \to Y \times Y$, $h \mapsto (th,sh)$ is the isotropy subbundle of $H \tto Y$, and our claim is that there exists $V \ni y_0$ open for which $a^{-1}(\Delta V) \subset k^{-1}(W)$. For each compact neighborhood $C$ of $y_0$ we can by PF1 find some continuous map $C \ftimes{\phi}{t} G \ftimes{s}{\phi} C \to H$ with image $H \mathbin| C = a^{-1}(C \times C)$, and since by the properness of $G \tto X$ the domain of this map is compact, so must be its image. The set $F^C := a^{-1}(\Delta C) \smallsetminus k^{-1}(W)$, being closed inside the compact set $H \mathbin| C$, must then itself be compact. Moreover, the intersection of all $F^C$, being contained in \[%
\textstyle%
	a^{-1}\bigl(\Delta\left.\bigcap C\right.\bigr) \smallsetminus k^{-1}(W)
		= a^{-1}(\Delta\{y_0\}) \smallsetminus k^{-1}(W)
		= H(y_0) \smallsetminus k^{-1}(W)
		= \emptyset,
\] must be empty. Thus, since the $F^C$ form a directed family under reverse inclusion, by the finite intersection property at least one of them must be empty and so there has to be some $C$ whose interior $V$ verifies our claim.

As to the closedness of $X_N$, let $\bar{x} \in \overline{X}_N$ be a point in its closure. By PF2, $\bar{x} = t\bar{g} = q(\bar{g},\bar{y})$ for some $(\bar{g},\bar{y}) \in G \ftimes{s}{\phi} Y$. It will then suffice to prove that $\bar{y}$ lies in $Y_N$. So, let $\bar{h} \in H(\bar{y})$ be given: we must show that $k(\bar{h}) \in N$. By our hypotheses on $G \tto X$, there exists some continuous map $g: U \to G$ defined in an open neighborhood $U$ of $\phi\bar{y}$ with $g(u) \in G(u)$ for all $u$ in $U$ and with $g(\phi\bar{y}) = \phi\bar{h}$. By PF1 as applied to the following maps of domain $V = \phi^{-1}(U) \ni \bar{y}$, \[%
	Y \times Y \longfrom V \longto G,~(v,v) \longmapsfrom v \longmapsto g(\phi v),
\] there exists some continuous map $h: V \to H$ such that $h(v) \in H(v)$ for all $v$ in $V$ and such that $\phi h(\bar{y}) = g(\phi\bar{y})$ and, hence, $h(\bar{y}) = \bar{h}$. Given $W \ni k(\bar{h})$ open, we can find $B \subset V$, $B \ni \bar{y}$ open for which $h(B) \subset k^{-1}(W)$. By openness of $q$, we have that $q(G \ftimes{s}{\phi} B)$ is an open neighborhood of $\bar{x}$ and thus intersects $X_N$. We can then find some $(g,y) \in G \ftimes{s}{\phi} B$ with $tg \in X_N$ and therefore some $y$ in $Y_N \cap B$, a fortiori, some $k\bigl(h(y)\bigr)$ in $N \cap W$. It follows that $k(\bar{h})$ lies in the closure of $N$ in $K$ and, consequently, in $N$ itself. \end{proof}

\begin{cor}\label{cor:appendix} Let\/ $G \tto X$ be as in the preceding proposition. Suppose\/ $X$ is locally connected, as well as connected. Let\/ $R: s^*E \simto t^*E \in \BB(G)$ be any representation of\/ $G \tto X$ on a locally trivial Banach bundle\/ $E \to X$ of locally finite rank. If\/ $R$ is trivial at\/ $x \in X$, in the sense that\/ $G(x) \subset \ker R$, then it is globally trivial viz.~trivial at every point of\/ $X$. \end{cor}

\begin{proof} By the finite-rank and the connectedness hypothesis $E \to X$ is automatically of class $C^0$ and $R$ is automatically uniformly continuous. By the local connectedness of $X$ we can find some atlas of local trivializations $U \times \R^n \simto E \mathbin| U$ of $E \to X$ with connected domains $U$. If $R$ is trivial at some $u$ in $U$ then by the above proposition as applied to \[%
\xymatrix@C=2.67em{%
	G \mathbin| U \ar@{<-}[r]^=
	\ar@<-.33ex>[d] \ar@<+.33ex>[d]
	&	G \mathbin| U \ar[r]^-{(R,\id)}
		\ar@<-.33ex>[d] \ar@<+.33ex>[d]
		&	\GL(\R^n) \times U
			\ar@<-.33ex>[d] \ar@<+.33ex>[d]
\\	U \ar@{<-}[r]^=
	&	U \ar[r]^=
		&	U
}\] we conclude that $R$ must be trivial at all points of $U$. Thus, the set of all $x$ at which $R$ is trivial is an open subset of $X$. But it must also be closed, for, by what we have just concluded, if $U$ contains a point in its complement then all of $U$ must be contained in that complement. Hence, in view of the connectedness of $X$, so long as it is nonempty, this set must be all of $X$. \end{proof}

\begin{exm}\label{exm:appendix} If $G \to \R$ is one of our nontrivial bundles of compact Lie groups with $G_x = \{1x\}$ trivial at $x = 0$, then $G \to \R$ cannot be told apart from the unit groupoid $\R \tto \R$ by looking only at its continuous representations on locally trivial Banach bundles of finite rank. \end{exm}

\paragraph*{Orbispace groupoids.} An \emph{orbispace groupoid} is a locally compact groupoid $G \tto X$ which, in the vicinity of each point $x \in X$, is equivalent as a topological stack to the transformation groupoid $K \times Z \tto Z$ associated with some compact group $K$ of transformations of a locally compact Hausdorff space $Z$ in the sense that over some open neighborhood $U$ of $x$ there exists a span of topological weak equivalences of the kind obtained upon substituting $G \mathbin| U \tto U$ for $G \tto X$ in \eqref{eqn:appendix}. [Cf.~the comments following our enunciation of the properties PF1–3 in section \ref{sec:images}.] The anchor map of an orbispace groupoid is of necessity always proper.

A classical theorem in the theory of compact groups \cite[p.~18, Thm.~4.2]{Bre72} asserts that if $H$ is a closed subgroup of a compact group $K$ and if $R: H \to \GL(\R^m)$ is any finite-dimensional representation of this subgroup then there exist representations $S: K \to \GL(\R^n)$ of dimension $n \geq m$ for which $R$ is contained in the restriction of $S$ to $H$ as a subrepresentation. Now, when $K$ is a group of transformations of a space $Z$ as above and $H = \stab(z)$ is the stabilizer of a point $z \in Z$, the result in question immediately entails a somewhat weaker variant of our local extension theorem for the transformation groupoid $K \times Z \tto Z$ in the finite-dimensional nonprojective case. (An observation to the same effect, although under unnecessarily restrictive hypotheses namely that $K$ be connected Lie, appears in \cite[Prop.~4.19]{Bos11}.) The stronger conclusions of our Theorem \ref{thm:local}, however, do not seem to be within reach of this line of argument except when $\stab(z) = K$ i.e.~$z$ is a fixed point for the $K$~action. On the other hand, our theorem solves the local extension problem for an arbitrary locally compact $K$ acting \emph{properly} on $Z$ in the sense that the corresponding transformation groupoid $K \times Z \tto Z$ is proper, a situation which evidently covers the case of a compact $K$.

In the differentiable setting all proper groupoids $G \tto X$ are orbispace groupoids, in fact, they are such in an even stronger sense [see our remarks at the beginning of section \ref{sec:local}]: for any given $x$, there exist weak equivalences $G \mathbin| U \xfrom{\phi} H \xto{\psi} K \times Z$ under which $U \ni x \xmapsfrom{\phi} y \xmapsto{\psi} z$ with $\stab(z) = K$. On the other hand, examples of nondifferentiable proper groupoids that do not enjoy the latter property abound: we have $\stab(w) \simeq G(u)$ whenever $u \xmapsfrom{\phi} v \xmapsto{\psi} w$, so since $\stab(w) \subset K = \stab(z)$ it must be possible to embed $G(u)$ into $G(x)$ as a subgroup for all $u$ in a neighborhood of $x$; but this cannot be done e.g.~for the nontrivial bundles $G \to \R$ of compact Lie groups which we discussed previously. Actually, our proposition enables us to say more:

\begin{cor}\label{cor:appendix*} Let\/ $G \tto X$ be any locally compact groupoid satisfying the conditions\/ {\upshape (ii)} of\/ {\upshape Proposition \ref{prop:appendix}} and having at least one trivial isotropy group\/ $G(x) = \{1x\}$. Suppose that\/ $X$ is both connected and locally connected. Then\/ $G \tto X$ cannot be an orbispace groupoid unless it is\/ \emph{principal} i.e.~has all its isotropy groups trivial. \end{cor}

\begin{proof} Connected open subsets $V$ of domains $U$ of weak equivalence spans $G \mathbin| U \xfrom{\phi} H \xto{\psi} K \times Z$ form a basis for the topology of $X$. By our proposition, every such $V$ must lie entirely either within the set of all $u$ for which $G(u) = \{1u\}$ or else within its complement. Our corollary now follows immediately from the connectedness of $X$. \end{proof}

\begin{exm}\label{exm:appendix*} The nontrivial bundles of compact Lie groups $G \to \R$ of Example \ref{exm:appendix} cannot be orbispace groupoids. Our proposition actually tells us that they cannot be “orbispaces” in any reasonable sense which would enable us by means of the theory of section \ref{sec:images} to carry over to them those representations that, in virtue of the above-mentioned classical theorem \cite[p.~18, Thm.~4.2]{Bre72}, exist on the transformation groupoids $K \times Z \tto Z$. Of course, this doesn't go against our saying in conformity with our local extension theorem that the local extension problem can be solved for all such (in fact, all proper) groupoids by other means, including directly, as is evidently the case for the simplest among our $G \to \R$. \end{exm}

%
\end{document}